\documentclass[11pt,reqno]{amsart}
\usepackage[usenames]{color}
\usepackage{amsmath,verbatim,graphicx,epstopdf,enumerate}

\usepackage[dvipsnames]{xcolor}

\definecolor{mypink1}{rgb}{0.858, 0.188, 0.478}
\definecolor{mypink2}{RGB}{219, 48, 122}
\definecolor{mypink3}{cmyk}{0, 0.7808, 0.4429, 0.1412}
\definecolor{mygray}{gray}{0.6}

\newcommand{\I}{\mathrm{i}}
\newcommand{\D}{{d}}

\newcommand{\wh}{\widehat}

\newcommand{\lb}{\left(}

\newcommand{\ve}{\varepsilon}

\newcommand{\rb}{\right)}
\newcommand{\PD}{\partial}

\newcommand{\Rb}{\mathbb{R}}
\newcommand{\Sb}{\mathbb{S}}

\newcommand{\Beq}{\begin{equation}}
\newcommand{\Eeq}{\end{equation}}
\newcommand{\beq}{\begin{equation*}}
\newcommand{\eeq}{\end{equation*}}
\newcommand{\bal}{\begin{align}}
\newcommand{\eal}{\end{align}}

\newcommand{\bp}{\begin{prob}}
	\newcommand{\ep}{\end{prob}}
\newcommand{\bpr}{\begin{proof}}
	\newcommand{\epr}{\end{proof}}

\newcommand{\bel}[1]{\begin{equation}\label{#1}}
\newcommand{\ee}{\end{equation}}

\newtheorem{theorem}{Theorem}[section]

\newtheorem{lemma}[theorem]{Lemma}
\newtheorem{proposition}[theorem]{Proposition}

\theoremstyle{definition}
\newtheorem{definition}[theorem]{Definition}
\newtheorem{remark}[theorem]{Remark}

\usepackage{amssymb, amsmath}
\title[Inverse problem for non-linear hyperbolic PDE]{Inverse Initial Boundary Value Problem for a Non-linear Hyperbolic Partial Differential Equation}
\author[Nakamura, Vashisth and Watanabe]{Gen Nakamura$^{\dagger}$, Manmohan Vashisth$^{\ddagger}$ and Michiyuki Watanabe$^{*}$} 
\address{$^{\dagger}$Department of Mathematics, Hokkaido University, Sapporo 060-0810, Japan.
	\newline\indent E-mail:{\tt \ nakamuragenn@gmail.com}}
\address{$^{\ddagger}$Beijing Computational Science Research Center, Beijing 100193, China.
	\newline
	\indent E-mail:{\tt\  manmohanvashisth@gmail.com}}

\address{$^{*}$Department of  Applied Mathematics,
	Faculty of Science,
	Okayama University of Science, Japan.
	\newline
	\indent E-mail:{\tt\  watanabe@xmath.ous.ac.jp}}
\begin{document}
	
	\begin{abstract}
		In this article we are concerned with an inverse initial boundary value problem for a non-linear wave equation in space dimension $n\geq 2$.  In particular we consider the so called interior determination problem. This non-linear wave equation has
		a trivial solution, i.e. zero solution. By linearizing this equation at the trivial solution, we have the usual linear wave equation with a time independent potential. For any small solution $u=u(t,x)$ of this non-linear equation, it is the perturbation of linear  wave equation  with time-independent potential perturbed by a divergence with respect to {$(t,x)$} of a vector whose components are quadratics with respect to $\nabla_{t,x} u(t,x)$. By ignoring the terms with smallness {$O(|\nabla_{t,x} u(t,x)|^3)$}, we will show that we can uniquely determine the potential and the coefficients of these quadratics by many boundary measurements at the boundary of the spacial domain over finite time interval {and the final overdetermination at $t=T$}. {In other words, the measurement is given by the so-called the input-output map (see \eqref{definition of NtD map}).}
	\end{abstract}
	\maketitle
	\section{Introduction}\label{Introduction}
	\setcounter{equation}{0}
	\renewcommand{\theequation}{1.\arabic{equation}}
	Let $\Omega\subset\mathbb{R}^{n}$\,($n\geq 2$) be a bounded domain with smooth boundary $\partial\Omega$. For $T>0$, let $Q_T:=(0,T)\times\Omega$ and denote its lateral boundary by $\partial Q_T:=(0,T)\times\partial\Omega$.
	
	Now consider the following initial boundary value problem (IBVP):
	\begin{align}\label{equation of interest}
	\begin{aligned}
	\begin{cases}
	\partial_{t}^{2}u(t,x){-\Delta  u(t,x)}+a(x)u(t,x)={\nabla_{t,x}\cdot\overrightarrow{C}(t,x,\nabla_{t,x}u(t,x))}, \  (t,x)\in Q_T, \\
	u(0,x)=\epsilon \phi(x),\,\, \partial_{t}u(0,x)=\epsilon \psi(x),  \quad  x\in\Omega,\\
	u(t,x)=\epsilon f(t,x), \quad (t,x)\in\partial Q_T,
	\end{cases}
	\end{aligned}
	\end{align}
	where {$\nabla_{t,x}:=(\PD_{0},\partial_1,\cdots,\partial_n),\,\,\PD_{0}=\PD_{t}$ and $\partial_j=\partial_{x_j}$ for $x=(x_1,\cdots,x_n)$}. 
	Here 
	$\overrightarrow{C}(t,x,q)$ 
	is given by 
	\begin{align}\label{definition of vector C}
	\begin{aligned}
	\overrightarrow{C}({t},x,q):=\overrightarrow{P}({t},x,q)+\overrightarrow{R}({t},x,q)
	\end{aligned}
	\end{align}
	with $q:=(q_0,\widetilde q)=({q_{0}},q_{1},\cdots,q_{n})\in\mathbb{C}^{{1+}n}$, we have
	\begin{align}\label{definition of P(x,q)}
	\overrightarrow{P}({t},x,q):=\lvert q\rvert^{2}{\vec{{b}}(t,x),}
	\end{align}
	where $|q|^2$ mostly denotes $\sum_{j={0}}^{n} q_j^2$, but 
	$$
	|q|^2=\sum_{j={0}}^n q_j\overline{q_j}
	$$
	for estimates. This is because there will be cases when $q\in \mathbb{C}^{1+n}$. The meaning of this $|q|^2$ will be clear from the context.  

	Denote by $B^\infty(\partial Q_T)$ the Fr\'echet space obtained by completing  $C^\infty(\partial Q_T):=\{f\big|_{\partial Q_T}: f\in C^\infty({\mathbb R}\times\partial Q)\}$ with respect to the metric $d_{\partial}(\cdot,\cdot)$ induced by the countable norms $$\displaystyle\sup_{t\in[0,T], 0\le k\le\ell}\Vert\partial_t^k g(t,\cdot) \Vert_{C^\ell(\partial\Omega)},\,\,\ell=0,1,\cdots.
	$$
	Further, let {$m\ge [n/2]+3$ where $[n/2]$ is the largest integer not exceeding $n/2$} and $B_M:=\{(\phi,\psi,f)\in C^\infty(\overline\Omega)^2\times B^\infty(\partial Q_T):\,d_{\partial}(0,f)+d(0,\psi)+ d(0,\phi)\le M\}$ with the metric $d(\cdot,\cdot)$ in the Fr\'echet space $C^\infty(\overline\Omega)$ induced by the countable number of norms $\Vert\cdot\Vert_{C^\ell(\overline\Omega)}$, $\ell=0,1,\cdots$ and a fixed constant $M>0$.
	
	We assume that $a\in C^{\infty}(\overline\Omega)$, $\vec{{b}}(t,x):=( b_{0},b_{1},b_{2},\cdots,b_{n})\in C_{(0)}^{\infty}\lb (0,T];C^{\infty}(\overline{\Omega})\rb$ is such that $\vec{b}(T,x)=0$ in $\Omega$ and $\overrightarrow{R}({t},x,q)\in C_{(0)}^\infty((0,T]; C^\infty(\overline\Omega\times H))$ with $H:=\{q=q_R+iq_I\in {\mathbb C}^{1+n}:\,q_R,\,q_I\in{\mathbb R}^{1+n},\,|q|\le h\}$ for some
	constant $h>0$ satisfying the following estimate: there exists a constant $C>0$ such that
	\begin{equation}\label{estimate of R vector}
	{\lvert\partial_{q}^{\alpha}\nabla_{t,x}^{\beta}\overrightarrow{R}(t,x,q)\rvert\leq C\lvert q\rvert^{3-\lvert\alpha\rvert}\, \,\mbox{for multi-indices $\alpha$\  with $\lvert\alpha\rvert\leq 3$ and $\beta$},}
	\end{equation}
	where
	$\partial_q=(\partial_{q_R},\PD_{q_{I}})$  and $C_{(0)}^\infty((0,T];E)$ is a set of a Fr\'echet space $E$ valued $C^\infty$ function over $[0,T]$ flat at $t=0$. 

	Then, there exists $\epsilon_0=\epsilon_0(h, T,m, M)>0$ such that
	\eqref{equation of interest} has a unique solution $u\in X_m:=X_m([0,T])$ for any $\lb \phi,\psi,f\rb\in B_M$ satisfying the compatibility condition of order $m-1$ and $0<\epsilon<\epsilon_0$, where $X_m(I):=\cap_{j=0}^m C^j(I; W^{m-j,2}(\Omega))$ for a time interval $I$ with the $L^2(\Omega)$ based Sobolev space $W^{m-j,2}(\Omega)$ of order $m-j$. We refer this by the {\sl unique solvability} of \eqref{equation of interest}.  In Section \ref{Epsilon expansion}, we will provide the proof of this together with the $\epsilon$-expansion of solution to \eqref{equation of interest}. Because of the presence of time-dependent coefficients in \eqref{equation of interest} and the space dimension $n\geq 2$, the proof of unique solvability and $\epsilon$-expansion of solution to \eqref{equation of interest} does not follow from  \cite{Dafermos,NW} in a straight forward manner.  The $\epsilon$-expansion in \cite{NW}, proved for one space dimension will work only for time-independent coefficient case. Hence in section \ref{Epsilon expansion}, by using the ideas from \cite{Dafermos,NW} and adding the several new arguments, we will prove the unique solvability together with the $\epsilon$-expansion for solution to \eqref{equation of interest}. 

Based on the unique solvability of \eqref{equation of interest}, define the {\sl input-output map}  ${\Lambda}_{\overrightarrow{C},a}^T$ by
	\begin{align}\label{definition of NtD map}
	{{\Lambda}_{\overrightarrow{C},a}^T(\epsilon \phi,\epsilon \psi,\epsilon f)=\lb \Big[\PD_{\nu}u^{\phi,\psi,f}+ \lb 0, \nu(x)\rb\cdot\overrightarrow{C}(t,x,\nabla_{t,x}u^{\phi,\psi,f})\Big]\Big|_{\partial Q_T}, u^{\phi,\psi,f}|_{t=T},\PD_{t}u^{\phi,\psi,f}|_{t=T}\rb, \,\,}
	\end{align}
	where $0<\epsilon<\epsilon_0,$ $u^{\phi,\psi,f}(t,x)$ for $(\phi,\psi,f)\in B_M$, is the solution to $\eqref{equation of interest}$ and $\nu(x)$ is the outward unit normal vector to $\partial\Omega$ at $x\in\partial\Omega$ directed into the exterior of $\Omega$.

	The inverse problem we are going to consider is the uniqueness of identifying the potential $a=a(x)$, and the quadratic nonlinearity $\overrightarrow{P}=\overrightarrow{P}(t,x,q)$ of $\overrightarrow{C}$ from the {input-output map}
	${\Lambda}^T_{\overrightarrow{C},a}$.  More precisely it is to show the following:
	
	\medskip
	\noindent
	$$
{	\Lambda_{\overrightarrow{C}^{(1)},a_{1}}^T=\Lambda_{\overrightarrow{C}^{(2)},a_{2}}^T
	\Longrightarrow	(a_1,\,\overrightarrow{P}^{(1)})=(a_2,\,\overrightarrow{P}^{(2)})\,\,\mbox{with}\,\,\overrightarrow{P}^{(i)}=|q|^2\vec{b}^{(i)},\,\,i=1,2}
	$$
	where
	${\Lambda}_{\overrightarrow{C}^{(i)},a_{i}}^T,\,i=1,2$ are the {input-output maps} given by $\eqref{definition of NtD map}$ for $\Big( a,\overrightarrow{C}\Big)=\Big( a_{i}, \overrightarrow{C}^{(i)}\Big),\,i=1,2$ and $(a_i,\overrightarrow{P}^{(i)}),\,i=1,2$ are $(a,\overrightarrow{P})$ associated to $\lb a_{i}, \overrightarrow{C}^{(i)}\rb\,i=1,2$.

	
	The non-linear wave equation of the form \eqref{equation of interest}
	{with the assumptions  $a(x)=0$, $\overrightarrow C(t,x,q)=\overrightarrow C(x,\widetilde q)$} arises as a model equation of a vibrating string with elasticity coefficient depending on strain and a model equation describing the anti-plane deformation of a uniformly thin piezoelectric material for the one spacial dimension (\cite{NWK}), and as a model equation for non-linear Love waves for the two spacial dimension (\cite{Rus}).
	
	There are several earlier works on inverse problems for some non-linear wave equations in one space dimension. For example, 
	Denisov \cite{Denisov quasilinear} considered identifying a nonlinear potential depending on the space variable and the derivative of the solution with respect to the space variable, Grasselli \cite{Grasselli} considered identifying the speed of a wave equation arising from the nonlinear vibration of elastic string with the nonlinearity given as the speed depending on the integration of the modulus of displacement over the string, and Lorenzi-Paparoni \cite{Lorenzi Paparoni} considered identifying a nonlinear potential given as some first order derivative of a function depending on the solution of the equation arising from the theory of absorption.
	
	Under the similar set up as our inverse problem except the space dimension { and with the assumptions $a(x)=0$, $\overrightarrow C(t,x,q)=\overrightarrow C(x,\widetilde q)$ for the equation}, authors in \cite{NW,NWK} identified the time-independent coefficients by giving a reconstruction formula in one space dimension which also gives uniqueness. We are going to prove the uniqueness for our inverse problem when the space dimension $n\geq 2$ and coefficient of non-linearity $\vec{b}$ is time-dependent. Authors in  \cite{Ali_Lauri_nonlienar-wave} studied the inverse problems of determining the potential from the source to solution map for a non-linear wave equation in Riemannian geometry. Recently \cite{Hintz_Uhlmann_Zhai} considered the inverse problems for determining the coefficients of non-linearities appearing in a semilinear wave equation on Lorentzian manifold. We refer to \cite{Oksanen_Yang-Mills,Isakov Book,Kian_nonlinear_wave,Kurylev_Lassas_Uhlmann-nonlinear hyperbolic,Lassas_ICM,Lassas_Uhlmann_Wang_semilinear wave,Lassas_Tony_semilinear_2020,Wang_Zou_non-linear hyperbolic} for more works on inverse problems related to non-linear hyperbolic equations.
	
	The physical meaning of our inverse problem can be considered as a problem to identify especially the higher order tensors in non-linear elasticity for its simplified model equation. In a smaller scale the higher order tensors become important. There is a recent uniqueness result proved in  \cite{hoop,Gunther_Jian_elastic} for some nonlinear isotropic elastic equation.
	
	We also point out some related works for elliptic and parabolic equations. For elliptic equations, Kang-Nakamura in \cite{KN} studied the uniqueness for determining the non-linearity in conductivity equation. Our result can be viewed as a generalization of \cite{KN} for non-linear wave equation with constant conductivity and a potential. There are other works related to non-linear elliptic PDE, we refer to \cite{CNV,Hervas and Sun, Isakov Uniqueness,Isakov and Nachman,Krupchyk_Uhlmann_partial data-gradient,Nakamura Sun,Sun Uhlmann,Sun Conjecture,Sun Semilinear}. Also, for nonlinear parabolic equations, we refer to \cite{Caro_Kian_parabolic,COY,Isakov Uniqueness,Klibanov}.

\medskip
Now we state the main result.
	\begin{theorem}\label{Uniqueness theorem}
		For $i=1,2$, let 
		\begin{align*}
		\overrightarrow{C}^{(i)}(t,x,q)=q+\overrightarrow{P}^{(i)}(t,x,q)+\overrightarrow{R}^{(i)}(t,x,q)  \ \mbox{and} \  \overrightarrow{P}^{(i)}(t,x,q)=\lvert q\rvert^{2} \vec{{b}}^{(i)}(t,x)
		\end{align*}
		with $\overrightarrow{P}^{(i)}\ \text{and} \ \overrightarrow{R}^{(i)},\,i=1,2$ satisfying the same conditions as for $\overrightarrow{P}\ \text{and}\ \overrightarrow{R}$. Further let $u^{(i)}\in X_m,\,i=1,2$ be the solutions to the following IBVP:
		\begin{align}\label{equation for ui}
		\begin{aligned}
		\begin{cases}
		\partial_{t}^{2}u^{(i)}(t,x)-{\Delta  u^{(i)}(t,x)}+a_{i}(x)u^{(i)}(t,x)={\nabla_{t,x}\cdot\overrightarrow{C}^{(i)}(t,x,\nabla_{t,x}u^{(i)}(t,x)),} \   (t,x)\in Q_T, \\
		u^{(i)}(0,x)=\epsilon\phi(x),\ \partial_{t}u^{(i)}(0,x)=\epsilon \psi(x),\ \  x\in\Omega,\\
		u^{(i)}(t,x)=\epsilon f(t,x), \quad (t,x)\in\partial Q_T
		\end{cases}
		\end{aligned}
		\end{align}
		with $\lb \phi,\psi,f\rb\in B_{M}$ satisfying the compatibility condition of order $m-1$ and any $0<\epsilon<\epsilon_0$.
		 Let ${\Lambda}_{\overrightarrow{C}^{(1)},a_{1}}^T$ and ${\Lambda}_{\overrightarrow{C}^{(2)},a_{2}}^T$  be the input-output maps as defined in $\eqref{definition of NtD map}$ corresponding to $u^{(1)}$ and $u^{(2)}$, respectively. Assume that $T$ is larger than the diameter of $\Omega$ and
		\begin{align}\label{equality of NtD map}
		{\Lambda}_{\overrightarrow{C}^{(1)},a_{1}}^T(\epsilon\phi,\epsilon\psi,\epsilon f)={\Lambda}_{\overrightarrow{C}^{(2)},a_{2}}^T(\epsilon\phi,\epsilon\psi,\epsilon f),\,\,(\phi,\psi,f)\in B_M,\,\,0<\epsilon<\epsilon_0.
		\end{align}
		Then we have
		\begin{align*}
		a_{1}(x)=a_{2}(x),\,\, x\in\Omega\,\,\mbox{and}\,\,\, \vec{b}^{(1)}(t,x)=\vec{b}^{(2)}(t,x),\,\,(t,x)\in Q_{T}.
		\end{align*}
	\end{theorem}
	\begin{remark}
	Note that in Theorem \ref{Uniqueness theorem} the coefficient of non-linearity $\vec{b}$ is time-dependent, hence our measurement is the input-output map which consists of the usual hyperbolic Dirichlet to Neumann map and the information of solution measured at the initial and final time. Due to this extra information of solutions, we can derive the integral identity given by \eqref{Integral identity before GO} in Section \ref{Proof of uniqueness Theorem}, for any solution $w$ to $\PD_{t}^{2}w-\Delta w+a(x)w=0$ and hence we can prove  Lemma \ref{Lemma about LI of solutions},  which immediately gives the  uniqueness of $\vec{b}(t,x)$. Recently there has been several works in the literature related to inverse problems for non-linear hyperbolic equations, but most of them showed the determination of time-independent coefficients of non-linearity from the boundary measurements. The determination of the time-dependent coefficients appearing in a non-linear hyperbolic partial differential equations from the boundary measurements has not been well studied in the prior works and this is the aim of the present work. To the best of our knowledge this is the first result which deals with the determination of the time-dependent coefficient of a quadratic non-linearity appearing in a non-linear hyperbolic partial differential equations from boundary measurements of the solution.
	\end{remark}
	The proof of Theorem \ref{Uniqueness theorem} will be done in two steps. Namely we first show that from $$\Lambda_{\overrightarrow{C}^{(1)},a_{1}}^T(\epsilon\phi,\epsilon\psi,\epsilon f)=\Lambda_{\overrightarrow{C}^{(2)},a_{2}}^T(\epsilon\phi,\epsilon\psi,\epsilon f),\,\,\,(\phi,\psi,f)\in B_M$$ we can have $a=a_1=a_2$ and $a$ can be reconstructed from one of $\Lambda_{a_i}^T$ which is the linearization of input-output map  $\Lambda_{\overrightarrow{C}^{(i)},a_{i}}$ defined in \eqref{definition of NtD map} and  given by
	$$
	\Lambda_{a_i}^T (\phi,\psi,f)={\lb \nu(x)\cdot\nabla_x u_1^{(i)\phi,\psi,f}\Big|_{\PD Q_{T}},u_{1}^{(i)}\Big|_{t=T},\PD_{t}u_{1}^{(i)}\Big|_{t=T}\rb},\,\ (\phi,\psi,f)\in B_{M}
	$$
	where $u_1^{(i)\phi,\psi,f}\in X_m$ is the solution to the initial boundary value problem:
	\begin{equation}\label{u_1^f}
	\begin{aligned}
	\begin{cases}
	\partial_t^2 v(t,x)-\Delta  v(t,x)+a_{i}(x)v(t,x)=0,\,\,(t,x)\in Q_T,\\
	v(0,x)=\phi(x),\,\,\PD_{t}v(0,x)=\psi(x),\,\,x\in\Omega,\\
	v(t,x)=f(t,x),\,\,(t,x)\in\partial Q_T.
	\end{cases}
	\end{aligned}
	\end{equation}
 Using the reconstruction for $a(x)$ and  varying the  initial and Dirichlet data for \eqref{u_1^f}, we can know the solution $v(t,x)$ of \eqref{u_1^f} in $Q_{T}$. Now using the uniqueness for $a_{i}$ in $\Omega$, we have the corresponding solutions to the linearized problem \eqref{u_1^f} are equal. This will help us to derive an integral identity involving $\vec{b}$.  Finally using \eqref{equality of NtD map} and  the special solutions for the linearized equation \eqref{u_1^f}, we prove the unique identification of $\vec{b}(t,x)$. We remark here that since our arguments for identifying $\vec{b}$ require the reconstruction for the lower order coefficient  therefore we have assumed that $a$ is time-independent. 

	\medskip
	The rest of this paper is organized as follows. In Section 2, we will introduce the $\epsilon$-expansion of the IBVP and  analyze the input-output map in $\epsilon$-expansion. As a consequence, we will show that the input-output map determines the input-output map  $\Lambda_{a}^{T}$  associated with the equation $\partial_t^2 v-\Delta_x v+av=0$ in $Q_{T}$. This immediately implies the uniqueness of identifying $a$. Section 3 is devoted to proving the uniqueness of identifying  { $\vec{b}(t,x)$, the coefficient of quadratic non-linearity with respect to $\nabla_{t,x}u$}.
	
	\section{$\epsilon$-expansion of the solution to the IBVP and input-output map in $\epsilon$-expansion}\label{Epsilon expansion}
	\setcounter{equation}{0}
	\renewcommand{\theequation}{2.\arabic{equation}}
	
	In this section, we provide the $\epsilon$-expansion, which will also provide the proof of the unique solvability of \eqref{equation of interest}. For the case when space dimension is one and the coefficients are time independent, there is a brief proof of the verification of the $\epsilon$-expansion given in \cite{NW}. Although in principle the idea of proof is the same as in \cite{NW}, it becomes more complicated and needs to add further arguments for the case when the space dimension becomes higher and the coefficients are time-dependent. Hence we will give the full proof verifying the $\epsilon$-expansion.
	
	\begin{theorem}\label{Unique solvability of direct problem non-linear}
		Let {$m\ge [n/2]+3$} and $(\phi,\psi,f)\in B_{M}$ 
		with a fixed constant $M>0$, then for given $T>0$, there exists $\epsilon_0=\epsilon_0(h, m, M)>0$ such that for any $0<\epsilon< \epsilon_{0}$,
		\eqref{equation of interest} has a unique solution 
		$u\in X_{m}$, where $h$,$B_M$ and  $X_m$ were defined in Section \ref{Introduction}. Moreover, it admits an expansion which we call $\epsilon$-expansion:
		\begin{equation}
		u=\epsilon u_{1}+\epsilon^{2} u_{2}+O(\epsilon^{3}),\,\,\ \epsilon\rightarrow 0,
		\end{equation}
		where $u_{1}$ is a solution to
		\begin{align}\label{equation for u1 non-linear}
		\begin{aligned}
		\begin{cases}
		\partial_{t}^{2}u_{1}(t,x)-\Delta u_{1}(t,x)+a(x)u_{1}(t,x)=0, \ (t,x)\in Q_{T},\\
		u_{1}(0,x)=\phi(x),\,\,\partial_{t}u_{1}(0,x)=\psi(x), \   x\in\Omega, \\
		u_{1}(t,x)=f(t,x),\   (t,x)\in\partial Q_{T},
		\end{cases}
		\end{aligned}
		\end{align}
		and $u_{2}$ is a solution to 
		\begin{align}\label{equation for u2 non-linear}
		\begin{aligned}
		\begin{cases}
		\partial_{t}^{2}u_{2}(t,x)-\Delta u_{2}(t,x)+a(x)u_{2}(t,x)={\nabla_{t,x}\cdot \lb \vec{b}(t,x)\lvert\nabla_{t,x}u_{1}(t,x)\rvert^{2}\rb}, \  (t,x)\in Q_{T},\\
		u_{2}(0,x)=\partial_{t}u_{2}(0,x)=0,\ x\in\Omega,\\
		u_{2}(t,x)=0,\ (t,x)\in \partial Q_{T}
		\end{cases}
		\end{aligned}
		\end{align}
		and  O$\lb\epsilon^{3}\rb$ means the following:
		\begin{align*}
		\begin{aligned}
		w(t, x)=\mbox{O}(\epsilon^{3})\Longleftrightarrow\Vert w\Vert_{X_m}:=\sup_{0\leq t\leq T}\lb\sum_{k=0}^{m}\lVert \overset{(k)}{w}(t,.)\rVert_{m-k}^{2}\rb^{1/2}=\mbox{O}(\epsilon^{3}),
		\end{aligned}
		\end{align*}
		where $\overset{(k)}{w}:=\frac{\partial^k  w}{\partial t^k}$ and $\Vert\cdot\Vert_{k}$ is the norm of the $L^{2}(\Omega)$ based Sobolev space $ W^{k,2}(\Omega)$ of order $k$. 
	\end{theorem}
	\begin{remark}\label{solvablility with nonzero D-data}${}$
		\begin{enumerate}
		\item \label{wellposedness for u1} For the well-posedness of initial boundary value problem \eqref{equation for u1 non-linear} we have the following. Let $m\in{\mathbb N}$ and let $\phi\in W^{m,2}(\Omega),\,\psi\in W^{m-1,2}(\Omega)$, $f\in \widetilde X_m$ satisfy the compatibility condition of order $m-1$. Then there exists a unique solution $u_1\in X_m$ to \eqref{equation for u1 non-linear} with the estimate
		$$ \Vert u_1\Vert_{X_m}\le C(\Vert\phi\Vert_{W^{m,2}(\Omega)}+\Vert\psi\Vert_{W^{m-1,2}(\Omega)}+\Vert f\Vert_{\widetilde X_m}),
		$$
		where $C>0$ is a general constant and $\Vert f\Vert_{\widetilde X_m}:=\sup_{0\le t\le T}\lb\sum_{j=0}^m\Vert\overset{(j)}f(t,\cdot)\Vert^2_{W^{m-j,2}(\partial\Omega)}\rb^{1/2}$ is the norm of $f$ in the space $\widetilde X_m:=\cap_{j=0}^m C^j([0,T];W^{m-j,2}(\partial\Omega))$. This can be proved by starting from $m=1$ given in Theorem 2.45 of \cite{KKL} and argue as in the arguments given from \eqref{key estimate} in Subsection \ref{Solvabilty for semilinear equation} to the end of this subsection. 
		\newline

		\item \label{wellposedness for u2} For the well-posedness of initial boundary value problem \eqref{equation for u2 non-linear} with a 
		general inhomogeneous term $F=F(t,x)$ instead of ${\nabla_{t,x}\cdot \overrightarrow{P}(t,x,\nabla_{t,x}u_{1})}$, we have the following. Let $m\in{\mathbb N}$ and let $F\in X_{m-1}$ satisfy the compatibility condition of order $m-1$. Then there exists a unique solution $u_2\in X_m$ to \eqref{equation for u2 non-linear} with the estimate
		$$\Vert u_2\Vert_{X_m}\le C\Vert F\Vert_{X_{m-1}},$$
		where $C>0$ is a general constant.
		This can be proved by referring \cite{Dafermos} and \cite{Evans}.
		\newline
	\item \label{wellposedness for u} The compatibility condition of order $m-1$ given in \eqref{wellposedness for u1} and \eqref{wellposedness for u2} are that considered independently from \eqref{equation of interest}. Nevertheless, in relation with \eqref{equation of interest}, if we want to have the solution $u$ of \eqref{equation of interest} to be in $ X_m$, then the compatibility conditions for both \eqref{equation for u1 non-linear} and \eqref{equation for u2 non-linear} are of the same order $m-1$ with $m\ge [n/2]+3$. This is due to the assumption we made for $\vec{b}$ and $\overrightarrow R$.
		\end{enumerate}
	\end{remark}

	
	Our strategy for the proof of Theorem \ref{Unique solvability of direct problem non-linear} is as follows: 
	\begin{itemize}
		\item We look for a solution $u(t,x)$ to \eqref{equation of interest} of the form
		\begin{equation}\label{form of u}
		\begin{aligned}
		u(t,x):=\epsilon\left\{u_{1}(t,x)+\epsilon\left(u_{2}(t,x)+w(t,x)\right)\right\},
		\end{aligned}
		\end{equation}
		where $u_1$, $u_2$ are the solutions to the initial boundary value problems \eqref{equation for u1 non-linear} and \eqref{equation for u2 non-linear}, respectively, 
		and derive the equation for $w$, which has the form 
		\[\partial_{t}^{2}w-B(w(t))w=\epsilon\mathcal{F}(t,x,{\nabla_{t,x}}w;\epsilon)\]
		{(see \eqref{Semilinear for w} and \eqref{mathcal F})}.
		\item For a given function $U(t)$, we prove the unique solvability of the semilinear wave equation of the form: 
		\[\partial_{t}^{2}w_{sem}-B(U(t))w_{sem}=\epsilon\mathcal{F}(t,x,{\nabla_{t,x}}w_{sem};\epsilon)\]
		with zero initial and boundary data.
		\item We prove that the map $T(U)=w_{sem}$ is a contraction mapping.
	\end{itemize}

\medskip
We first derive the equation for $w$. Direct computations show that  $w(t,x)$ has to satisfy
	\begin{align}
	\begin{aligned}\label{equation of w(t,x)}
	\begin{cases}
	\partial_{t}^{2}w-\Delta w+a(x)w=\epsilon^{-2}\nabla_{t,x}\cdot \overrightarrow{R}(t,x,\epsilon\nabla_{t,x}u_{1}+\epsilon^{2}\nabla_{t,x}u_{2}+\epsilon^{2}\nabla_{t,x}w)\\
	\ \ \ 	{+2\epsilon \nabla_{t,x}\cdot\lb \nabla_{t,x}u_1\cdot \nabla_{t,x}u_{2}\,\vec{b}\rb+2\epsilon \nabla_{t,x}\cdot\lb \nabla_{t,x}u_{1}\cdot \nabla_{t,x}w\,\vec{b}\rb}\\
	\ \ \ {+\epsilon^{2}\nabla_{t,x}\cdot\lb \lb \lvert\nabla_{t,x}u_{2}\rvert^{2}+2\nabla_{t,x}u_{2}\cdot\nabla_{t,x}w+\lvert\nabla_{t,x}w\rvert^{2}\rb\vec{b}\rb},\,\,\,  (t,x)\in Q_{T},\\
	w(0,x)=\partial_{t}w(0,x)=0,\ x\in\Omega,\\ \ w(t,x)=0,\,\, (t,x)\in\PD Q_{T}.
	\end{cases}
	\end{aligned}
	\end{align}
	By the mean value theorem, we have
	\begin{align*}
	\overrightarrow{R}(t,x,\epsilon\nabla_{t,x}u_{1}+\epsilon^{2}\nabla_{t,x}u_{2}+\epsilon^{2}\nabla_{t,x}w)&=\overrightarrow{R}(t,x,\epsilon\nabla_{t,x}u_{1}+\epsilon^{2}\nabla_{t,x}u_{2})\\
	&\ \ \ \ \ +\int\limits_{0}^{1}\frac{d}{d\theta}\overrightarrow{R}(t,x,\epsilon\nabla_{t,x}u_{1}+\epsilon^{2}\nabla_{t,x}u_{2}+\theta\epsilon^{2}\nabla_{t,x}w)d\theta\\
	&=\overrightarrow{R}(t,x,\epsilon\nabla_{t,x}u_{1}+\epsilon^{2}\nabla_{t,x}u_{2})+\epsilon^{3} K(t,x,\epsilon\nabla_{t,x}w;\epsilon)\nabla_{t,x}w,
	\end{align*}
	where 
	\begin{align*}
	\epsilon K(t,x,\epsilon\nabla_{t,x}w;\epsilon):=\int\limits_{0}^{1}\nabla_q\overrightarrow{R}(t,x,\epsilon\nabla_{t,x}u_{1}+\epsilon^{2}\nabla_{t,x}u_{2}+\theta\epsilon^{2}\nabla_{t,x}w)d\theta
	\end{align*}
	with $\nabla_q\overrightarrow{R}(x,q)=\left(\partial_{q_{j}}R_{i}\right)_{{0}\leq i,j\leq n}$ and $K=(K_{ij})$ with $K_{ij}=\partial_{q_{j}}R_{i}$. 
	
	Introduce the following notations:
	\begin{equation}\label{Notations}
	\begin{aligned}
	\begin{cases}
	\epsilon F(t,x,\nabla_{t,x} u_1,\nabla_{t,x}u_2;\epsilon):=2\epsilon\nabla_{t,x}\cdot((\nabla_{t,x} u_1\cdot\nabla_{t,x} u_2)\vec{b})+\epsilon^2\nabla_{t,x}\cdot(|\nabla_{t,x} u_2|^2\vec{b})\\
	\qquad\qquad\qquad\qquad\qquad\quad\quad\,\,\,+\epsilon^{-2}\nabla_{t,x}\cdot\overrightarrow{R}(t,x,\epsilon\nabla_{t,x} u_1+\epsilon^2\nabla_{t,x} u_2),\\
	\epsilon\Gamma(t,x,\nabla_{t,x} w;\epsilon):=2\epsilon(\vec{b}\otimes\nabla_{t,x} u_1)+2\epsilon^2(\vec{b}\otimes\nabla_{t,x} w)+2\epsilon^2(\vec{b}\otimes\nabla_{t,x} u_2)\\
	\quad\qquad\qquad\qquad\qquad+\epsilon K(t,x,\epsilon\nabla_{t,x} w;\epsilon)+\epsilon^2\mathcal{K}(t,x,\epsilon\nabla_{t,x}w;\epsilon)\nabla_{t,x}w,\\
	\epsilon \overrightarrow{G}(t,x,\nabla_{t,x} w;\epsilon)\cdot\nabla_{t,x}z:=2\epsilon(\nabla_{t,x}^2u_1)\cdot(\vec{b}\otimes\nabla_{t,x}z)+2\epsilon(\nabla_{t,x}\cdot\vec{b})(\nabla_{t,x}u_1\cdot\nabla_{t,x}z)\\
	\quad\qquad\qquad\qquad\qquad+\epsilon^2(\nabla_{t,x}\cdot\vec{b})(\nabla_{t,x}w\cdot\nabla_{t,x}z)+2\epsilon^2(\nabla_{t,x}^{2}u_2)\cdot(\vec{b}\otimes\nabla_{t,x}z)\\
	\quad\qquad\qquad\qquad\qquad+\epsilon^2(\nabla_{t,x}\cdot\vec{b})(\nabla_{t,x}u_{2}\cdot\nabla_{t,x}z)+\epsilon(\nabla_{t,x}\cdot K)\cdot\nabla_{t,x}z,
	\\ 
	B(w)z:=\Delta z-a(x)z+\epsilon\Gamma(t,x,\nabla_{t,x} w;\epsilon)\cdot\nabla_{t,x}^2 z,
	\end{cases}
	\end{aligned}
	\end{equation}

	\noindent
	where \textquotedblleft $\cdot$\textquotedblright={real inner product}, $\otimes=\mbox{tensor product}$,  ${\nabla^2_{t,x}} w=\mbox{Hessian of $w$}$, the $j$-th component of $(\nabla_{t,x}\cdot K)$ is $\sum_{i=0}^n \partial_i K_{ij}$ and
	the $(i,j)$-component of $\mathcal{K}\nabla_{t,x}w$ is $\sum_{l=0}^n\partial_{q_j}K_{il}\partial_l w$. Notice here that $\partial_i$ in $\sum_{i=0}^n \partial_i K_{ij}$ is just acting to the $x_i$ variable of $K_{ij}(t,x,q;\epsilon)$. Also $\Gamma(t,x,\nabla_{t,x};\epsilon)\cdot\nabla_{t,x}^2 w$ is the inner product of the two matrices $\Gamma(t,x,\nabla_{t,x};\epsilon)$ and $\nabla_{t,x}^2 w$.
	
	Then \eqref{equation of w(t,x)} can be written in the following form:
	\begin{align}\label{equation for w in final form}
	\begin{aligned}
	\begin{cases}
	\partial_{t}^{2}w-{B}(w)w-\epsilon \overrightarrow{G}(t,x,\epsilon\nabla_{t,x}w;\epsilon)\cdot\nabla_{t,x}w=\epsilon F(t,x,\nabla_{t,x}u_{1},\nabla_{t,x}u_{2};\epsilon)\,\,\,\text{in $Q_T$},\\
	w(0,x)=\partial_{t}w(0,x)=0\,\,\mbox{in}\,\, \Omega\,\,\text{and}\,\,w(t,x)=0\,\,\mbox{on}\,\, \partial Q_{T}.
	\end{cases}
	\end{aligned}
	\end{align}
	
	Now to complete the proof of Theorem \ref{Unique solvability of direct problem non-linear} it is enough to prove the following. 
	\begin{theorem}\label{Existence of w epsilon}
		Let $m\ge [n/2]+3$ and $(\phi,\psi,f)\in B_M$ satisfying the compatibility condition of order $m-1$. Then, for given $T>0$  there exists $\epsilon_0=\epsilon_0(h, m, M)>0$ and $w=w(t,x;\epsilon)\in X_m$ for $0<\epsilon<\epsilon_0$ such that each $w=w(\cdot,\cdot;\epsilon)$ is the unique solution to the initial boundary value problem
		\eqref{equation for w in final form} with the estimate
		\begin{equation}\label{estimate of w }
		\Vert w\Vert_{X_m}=O(\epsilon)\,\,\,\text{\rm as $\epsilon\rightarrow 0$}.
		\end{equation}
		\end{theorem}
	
	\bigskip
In order to prove this, 
let $Z(M)$ with $M>0$ be the set of $U$ satisfying
\begin{align}\label{Condition on U}
	\begin{aligned}
	\begin{cases}
	U\in X_m=X_m([0,T]),\\
	U(0,x)=\partial_t U(0,x)=0,\,\,x\in\Omega,\\
	U(t,x)=0,\,\,(t,x)\in {\partial Q_{T}},\\
    \Vert U\Vert_{X_m}\le M.
	\end{cases}
	\end{aligned}
	\end{align}

\noindent
Then based on the aforementioned strategy of proof of Theorem \ref{Unique solvability of direct problem non-linear}, we consider for $U\in Z(M)$ the following semilinear wave equation corresponding to the equation \eqref{equation for w in final form}
	\begin{equation}\label{Semilinear for w}
	\begin{aligned}
	\begin{cases}
	\partial_{t}^{2}w_{sem}-B(U)w_{sem}=\epsilon\mathcal{F}(t,x,\nabla_{t,x}w_{sem};\epsilon),\,\,(t,x)\in Q_{T},\\
	w_{sem}(0,x)=\partial_{t}w_{sem}(0,x)=0,\,\, x\in\Omega,\\
	w_{sem}(t,x)=0,\,\, (t,x)\in \partial Q_{T},
	\end{cases}
	\end{aligned}
	\end{equation}
	where 
	\begin{equation}\label{mathcal F}
	\begin{aligned}
	\mathcal{F}(t,x,h;\epsilon):= F(t,x,\nabla_{t,x}u_{1},\nabla_{t,x}u_{2};\epsilon)+\overrightarrow{G}(t,x,h;\epsilon)\cdot h\\
	\end{aligned} 
	\end{equation}
	
	\medskip

	\subsection{Unique solvabilty for the semilinear wave equation \eqref{Semilinear for w}} \label{Solvabilty for semilinear equation}${}$
	\par
	
	In this subsection, we give a proof of the following unique solvability for the semilinear wave equation \eqref{Semilinear for w}. 
\begin{proposition}\label{Existence of solution to semi-linear equation}
		Let {$m\geq[n/2]+3$} be an integer. Then, there exists $\epsilon_{1}>0$ such that the initial boundary value problem \eqref{Semilinear for w} has a unique solution $w_{sem}\in Z(M)$ for each $0<\epsilon<\epsilon_1$
		with the estimate
		\begin{equation}\label{Estimate for w semi-linear}
		\Vert w_{sem}\Vert_{X_m}\le \epsilon Ce^{K T},\,\,
		0<\epsilon<\epsilon_1,
		\end{equation}
		where $C$ and $K$ are positive constants depending on $M$ and $\epsilon_1$.

\end{proposition}
	
It is convenient to introduce the following notations for the proof of Proposition \ref{Existence of solution to semi-linear equation}. 
We first introduce the notation $\widetilde{B}(U)$. From \eqref{Notations} we have \begin{align*}
	    \begin{aligned}
	    B(U)w&=\Delta w-aw+\epsilon\Gamma(t,x,\nabla_{t,x}U;\epsilon)\cdot\nabla_{t,x}^{2}w\\&=\Delta w-aw+\epsilon\Gamma_{00}(t,x,\nabla_{t,x}U;\epsilon)\PD_{t}^{2}w\\
	    &\,\,+\epsilon\sum_{j=1}^{n}\big(\Gamma_{0j}(t,x,\nabla_{t,x}U;\epsilon)+\Gamma_{j0}(t,x,\nabla_{t,x}U;\epsilon)\big)\PD_{tx_{j}}^{2}w\\
	    & \ \ +\epsilon\sum_{1\leq i\leq n,\ 1\leq j\leq n}\Gamma_{ij}(t,x,\nabla_{t,x}U;\epsilon)\PD^{2}_{x_{i}x_{j}}w\\
	    &=:\epsilon\Gamma_{00}(t,x,\nabla_{t,x}U;\epsilon)\PD_{t}^{2}w+\widetilde{B}(U)w,
	    \end{aligned}
	\end{align*}
where $\Gamma_{ij}$ stands for $(i,j)$ component of the matrix $\Gamma$ and $\partial_{tx_j}:=\partial_t\partial_{x_j}=\partial_0\partial_j$. Note that the indices $0j,\,j0$ of $\Gamma_{0j}, \Gamma_{j0}$ correspond to $\partial_{tx_j}=\partial_0\partial_j=\partial_j\partial_0$. We further introduce the notations $A_U(t),\,L
$ and some other notations. Namely, denote by 
\begin{equation}\label{def AU(t)}
A_{U}(t):=\Big( 1-\epsilon\Gamma_{00}(t,x,\nabla_{t,x}U,\epsilon)\Big)^{-1}\widetilde{B}(U(t))w
\end{equation}
then  
\begin{equation}\label{de L}
Lw:=\Big( 1-\epsilon\Gamma_{00}(t,x,\nabla_{t,x}U,\epsilon)\Big)^{-1}\Big( \PD_{t}^{2}w-B(U(t))w\Big)=\PD_{t}^{2}w-A_{U}(t)w.
\end{equation}
Also let $\|\cdot \|_{m}$ be the norm of the space $W^{m,2}(\Omega)$ and let $W_0^{1,2}(\Omega):=\overline{C_0^{\infty}(\Omega)}^{\Vert\cdot\Vert_0}$ with the space $C_0^\infty(\Omega)$ of infinitely differentiable functions with compact support in $\Omega$. Further we write $\PD_t u = \dot{u}$ and $\PD_t^m u=\overset{(m)}{u}$.
 
		
	\medskip
	We first prove several lemmas which will lead us to give the proof of Proposition \ref{Existence of solution to semi-linear equation}.
	\begin{lemma}\label{elliptic part}
		Let $U$ satisfy \eqref{Condition on U} and restrict $\ve$ to vary in $[0,\ve_0]$ for a fixed small $\ve_0>0$. Then 
		$A_{U}(t)$ has the following properties. 
		\begin{enumerate}[(1)]
			%
			\item There is a constant $\nu >0$ such that 
			\begin{equation}\label{eqn:elliptic-estimate}
			\| v \|_{k+1}\le \nu (\| v \|_{k-1} + \| A_U(t)v \|_{k-1} ), 
			\quad k=0, \cdots ,m-2, 
			\end{equation}
			for any $v \in W_0^{1,2}(\Omega)\cap W^{k+{1},2}(\Omega)$ 
			and $t\in [0,T]$.  
			\item The {coercivity} holds for $A_U$. That is there are positive constants $\chi$, $\lambda$ such that 
			\begin{align}\label{coercivity}
			& -\Big\langle A_U(t)v, v\Big\rangle+\chi \| v\|^2_{0} \ge 
			\lambda \| v\|_1^2, 
			\,\,\,t\in [0,T],\,\,\text{real valued}\,\,v\in W_0^{1,2}(\Omega)\,\, 
			\end{align}
			with the continuous extension of $L^{2}(\Omega)$ inner product giving the pairing $\Big\langle \cdot,\cdot\Big\rangle$ between $W^{-1,2}(\Omega)$ and $W_0^{1,2}(\Omega)$.
			\item There is a continuous function 
			$\sigma :[0,\infty)\times [0,\infty) \to [0,\infty)$ 
			such that for every $M>0$ 
			and every $U$, $\overline{U} \in W^{1,2}(\Omega)$ with 
			{$[ U(t)]_{1}$, $[ \overline{U}(t)]_1 \le M$}, we have 
			\begin{equation}\label{estimate on A_U}
			\| (A_U-A_{\overline{U}}) w \|_{0}\le 
			\epsilon \sigma (M,\epsilon)\cdot 
			\Vert\nabla_{t,x} (U(t)- \overline{U}(t))\Vert_0,\,\,\,{ t\in[0,T]}
			\end{equation}
			for $w \in Z(M)$.
		\end{enumerate}
	
		\begin{proof}
			First of all we note that for (1), (2), the terms in $A_U(t)$ which have $\partial_t$ do not contribute because $v$ is independent of $t$. Then by using the standard elliptic regularity argument, we can have the properties (1) and (2).  
			As for the property (3), we divide $(A_U-A_{\overline U})w$ into three parts. That is by using the definitions of $A_U(t)$, we start estimating $\|(A_U-A_{\overline{U}})w\|_{0}$
			as follows. 
    \begin{equation}\label{estimate of AU-AbarU}
    \begin{array}{ll}
     \|(A_U-A_{\overline{U}})w\|_{0}&=\Bigg\|\left[\Big(1-\epsilon\Gamma_{00}(\nabla_{t,x}U)\Big)^{-1}\widetilde{B}(U)-\Big(1-\epsilon\Gamma_{00}(\nabla_{t,x}\overline{U})\Big)^{-1}\widetilde{B}(\overline{U})\right]w\Bigg\|_{0}\\
     &\le H_1+H_2+H_3,
     \end{array}
    \end{equation}
    where we have suppressed the variables $t,x,\epsilon$ in $\Gamma_{00}(t,x,\nabla_{t,x}U,\epsilon)$ and $H_j,\,j=1,2,3$ are defined as
    \begin{equation}
    \begin{array}{ll}
     H_1:=\Vert (1-\epsilon\Gamma_{00}(\nabla_{t,x} U))^{-1}(1-\epsilon\Gamma_{00}(\nabla_{t,x}\overline  U))^{-1}\big (\widetilde B(U))-\widetilde B(\overline U)\big )w\Vert_0,\\
     H_2:=\epsilon\Vert (1-\epsilon\Gamma_{00}(\nabla_{t,x} U))^{-1}(1-\epsilon\Gamma_{00}(\nabla_{t,x}\overline  U))^{-1}\Gamma_{00}(\nabla_{t,x}U)\big(\widetilde B(\overline U)-\widetilde B(U)\big)w\Vert_0,\\
     H_3:=\epsilon\Vert (1-\epsilon\Gamma_{00}(\nabla_{t,x} U))^{-1}(1-\epsilon\Gamma_{00}(\nabla_{t,x}\overline  U))^{-1}\big(\Gamma_{00}(\nabla_{t,x}U)-\Gamma_{00}(\nabla_{t,x}\overline U)\big)\widetilde B(U)w\Vert_0.
    \end{array}
    \end{equation}
    
    In order to estimate $H_j,\,j=1,2,3$ we introduce the following notation. That is for any matrix $Q=(Q_{ij})$, we define a matrix $Q^\natural=(Q_{ij}^\natural)$ with $Q_{ij}^\natural$ defined as
    \begin{equation}
    Q_{ij}^\natural=
    \left\{
    \begin{array}{ll}
   0,\,\,& i=0, j=0,\\
   Q_{ij},\,\,& \text{otherwise}.
    \end{array}
    \right.
    \end{equation}
    We will only give how to estimate $H_1$ because $H_2, H_3$ can be estimated similarly. By the definition of $\widetilde B(U)$ we have
    \begin{equation}
    \begin{array}{ll}
    \Vert\big(\widetilde{B}(U)-\widetilde{B}(\overline U)w\Vert_0= 2\epsilon^2\Vert (\overrightarrow b\otimes\nabla_{t,x}U)^\natural\cdot\nabla_{t,x}^2w-(\overrightarrow b\otimes\nabla_{t,x}\overline U)^\natural\cdot\nabla_{t,x}^2w\Vert_0\\
    \qquad+\epsilon\Vert\big(K^\natural(\epsilon\nabla_{t,x} U)-\big(K^\natural(\epsilon\nabla_{t,x}\overline U)\big)\nabla_{t,x}^2w\Vert_0+\epsilon^2\Vert\big((\mathcal{K}(\nabla_{t,x}U)\nabla_{t,x}U)^\natural-(\mathcal{K}(\nabla_{t,x}\overline U)\nabla_{t,x}U))^\natural\big)\nabla_{t,x}^2w\Vert_0\\
    \qquad=: I_1+I_2+I_3,
    \end{array}
    \end{equation}
    where
    \begin{equation}
    I_1:=2\epsilon\Vert (\overrightarrow b\otimes\nabla_{t,x}U)^\natural\cdot\nabla_{t,x}^2w-(\overrightarrow b\otimes\nabla_{t,x}\overline U)^\natural\cdot\nabla_{t,x}^2w\Vert_0\le2\epsilon(\Vert\partial_t w\Vert_1+\Vert w\Vert_2)\Vert\overrightarrow b \Vert_{L^\infty(\Omega)}\Vert\nabla_{t,x}(U-\overline U)\Vert_0    
    \end{equation}
    and due to \eqref{estimate of R vector} we have
    \begin{equation}
    \begin{array}{ll}
    I_2:=\epsilon\Vert \big(K^\natural(\epsilon\nabla_{t,x}U)-K^\natural(\epsilon\nabla_{t,x}\overline U)\big)\nabla_{t,x}^2w\Vert_0\\
    \quad\,\,\,=\epsilon\left(\int\limits_{\Omega}\big|\int\limits_{0}^{1}\{\nabla_q\overrightarrow R(\epsilon^2\theta\nabla_{t,x}U)^\natural-\nabla_q\overrightarrow R(\epsilon^2\theta\nabla_{t,x}\overline U)^\natural\}\,d\theta\,\nabla_{t,x}^2w\big|^2\,dx\right)^{1/2}\le C_M'\epsilon^3\Vert\nabla_{t,x}(U-\overline U)\Vert_0^2,\\
    I_3:=\epsilon^2\Vert\big((\mathcal{K}(\nabla_{t,x}U)\nabla_{t,x}U)^\natural-(\mathcal{K}(\nabla_{t,x}\overline U)\nabla_{t,x}U) U)^\natural\big)w\Vert_0
    \le C_M'\epsilon^3\Vert\nabla_{t,x}(U-\overline U)\Vert_0
    \end{array}
    \end{equation}
for some constant $C_M'>0$ depending only on $M$. Here
the estimate of $I_3$ can be obtained similarly as that of $I_2$.

\medskip    
Now we recall the following estimate similar to the one given in Theorem 7.2 of \cite{Mizohata Book} as a lemma.
\begin{lemma}\label{estimate in Mizohata book}
Let $m\ge[n/2]+3$ and $\kappa>0$ be the Sobolev embedding $W^{[n/2]+1,2}(\Omega)\hookrightarrow   C^0(\overline\Omega)$ constant. For a given $C^{m-1}$ function  $f(t,x;z)$ on 	$\mathcal{Q}:=\{(t,x,z)\in [0,T]\times \Omega \times {\bf C}:\,|z|\le \kappa M\}$, we have
\begin{equation}\label{well known estimate}
\| f(t,\cdot ;z)\|_{m-1}\le C_{m-1} M_{m-1}\Big\{ 1+
		\Big(1+\|z(t)\|_{m-2}^{m-2}\Big)\|z(t)\|_{m-1} \Big\},\,\,\,{t\in[0,T]}
\end{equation}
for each integer $m\geq [n/2]+3$, where
\[M_{m-1}:=\max_{|\beta | \le m-1}\sup_{\mathcal{Q}} \left|
		\left(\frac{\PD}{\PD x_1}, \cdots,\frac{\PD}{\PD x_n},\frac{\PD}{\PD z}\right)^{\beta} 
		f(t,x;z) \right|\] 
and a general constant $C_{m-1}$ depending on $m-1$.
\end{lemma}
\begin{proof}
The inequality given in Theorem 7.2 of \cite{Mizohata Book} is for the case $\Omega={\mathbb R}^n$. Nevertheless its argument of proof can be carry over to have \eqref{well known estimate} by noticing the following fact due to the existence of the extension operator $\mathcal{E}: W^{s,2}(\Omega)\rightarrow W^{s,2}({\mathbb R}^n)$ for $s\ge0$ coming from the $C^\infty$ smoothness of $\partial\Omega$. 
$W^{s,2}(\Omega)=H^s(\Omega)\,\,\text{for}\,\,s\ge0$ with equivalence of norms of these spaces, where $H^s(\Omega):=\{\phi|_\Omega: \phi\in H^s({\mathbb R}^n)\}$ with the norm
$\Vert\varphi\Vert_{H^s(\Omega)}:=\min\{\Vert\phi\Vert_{H^s({\mathbb R}^n)}: \phi|_\Omega=\varphi,\,\,\phi\in H^s({\mathbb R}^n)\}$
and $H^s({\mathbb R}^n):=\{\phi\in L^2({\mathbb R}^n):\Vert\phi\Vert_{H^s({\mathbb R^n})}:=\big(\int_{{\mathbb R}^n} (1+|\xi|^2)^s |\widehat\phi(\xi)|^2 d\xi\big)^{1/2}<\infty\} $ with the Fourier transform $\widehat\phi(\xi)$
of $\phi$ (see page 77 of \cite{McLean}).
We emphasize here that the norm of $\vert\varphi\Vert_{H^s(\Omega)}$ is given as the minimum of $\Vert\phi\Vert_{H^s({\mathbb R}^n)}$.
\end{proof}

By Lemma \ref{elliptic part}, we have $\Vert \big(1-\epsilon\Gamma_{00}(\nabla_{t,x}U)\big)^{-1}\Vert_{L^\infty(\Omega)}\le C_M''$ with some constant $C_M''>0$ depending only on $M$. Summing up these estimate we have
$$
H_1\le\epsilon\sigma_1(M,\epsilon)\Vert\nabla_{t,x}(U-\overline U)\Vert_0,
$$
where $\sigma_1(M,\epsilon)$ is defined likewise $\sigma(M,\epsilon)$. This finishes the proof of Lemma \ref{elliptic part}.

		\end{proof}
	\end{lemma}

	The following Lemma follows from an estimate similar to \eqref{well known estimate}.
	
	\begin{lemma}\label{lem:composite}
	Let
	\begin{equation}\label{mathcal F}
	    \widetilde{\mathcal{F}}(t,x,\nabla_{t,x}w):=\Big( 1-\epsilon\Gamma_{00}(t,x,\nabla_{t,x}U,\epsilon)\Big)^{-1} {\mathcal{F}}\lb t,x, {\nabla_{t,x}} w\rb.
	\end{equation}
Assume that $w\in Z_M$. 
If $m\ge [n/2]+3$,
		then we have 
		\[[\widetilde{\mathcal{F}}(\cdot ,t, {\nabla_{t,x}}w;\epsilon)]_{m-1}\le C
		(1+[w(t)]_m^{m-1}),
		\,\,\, t\in[0,T],\]
where
\begin{equation}\label{eqn:energy}
[w(t)]_m^2:=\displaystyle\sum_{j=0}^m\Vert\partial_t^j w(t)\Vert_{m-j}^2
\end{equation}
and $C>0$ is a general constant depending only on $M$.
\end{lemma}
\begin{proof}
By an argument similar to deriving \eqref{well known estimate}, we have
$$
 [\widetilde{\mathcal{F}}(\cdot ,t,z;\epsilon)]_{m-1}\le C_{m-1}M_{m-1}\left\{1+\Big( 1+[z(t)]_{m-2}^{m-2}\Big)[z(t)]_{m-1}\right\}
$$
with constants $C_{m-1}, M_{m-1}$ as in \eqref{well known estimate}.
This is because the space $X_r$ with non-negative integer $r$ has the property
$$
z\in X_r\Longrightarrow \partial_{t,x}^\alpha\, z\in X_{r-|\alpha|}
$$
for any multi-index $\alpha$ such that $|\alpha|\le r$.
Observe that
\begin{align*}
1+\Big( 1+[z(t)]_{m-2}^{m-2}\Big)[z(t)]_{m-1} & \le 
1+\Big( 1+[z(t)]_{m-1}\Big)^{m-2}\Big( 1+[z(t)]_{m-1}\Big) \\
& = 1+\Big( 1+[z(t)]_{m-1}\Big)^{m-1} \\
& \le C \Big( 1+[z(t)]_{m-1}^{m-1} \Big)
\end{align*}
with a general constant $C>0$.
Then by taking $z={\nabla_{t,x}} w$, we obtain the desired estimate. 
\end{proof}

\begin{lemma}\label{lem:linear wave}
For $S\in X_{m-1}$ consider the following initial boundary value problem
\begin{equation}\label{eqn:linear-wave}
		\begin{cases}
		L[v]= S\,\,\text{in}\,\, Q_{T},\\
		v(0,x)=0,\ \PD_t v(0,x)=0, \ \text{ $x\in\Omega$},\\
		v(t,x)=0,\ (t,x)\in\PD Q_{T}.
		\end{cases}
\end{equation}
If $(0,0,S)$ satisfies the compatibility condition of order $m-1$, there exists a unique solution $v\in X_m$ to \eqref{eqn:linear-wave} with the energy estimate
\begin{equation}\label{eqn:energy}
[v(t)]_m^2\le
C_m\int\limits_{0}^{T}[S(t)]_{m-1}^2\,dt,\,\,t\in[0,T],
\end{equation}
where $C_m>0$ is a general constant depending on $m$.
\end{lemma}
\begin{proof}
By using Lemma \ref{elliptic part} and handling the terms of $L$ with mixed derivatives $\partial_t\partial_{x_j},\,1\le j\le n$ by integration by parts using the boundary condition to derive an energy estimate, 
it follows from the standard argument that there exists a unique solution to \eqref{eqn:linear-wave} which satisfies the energy estimate \eqref{eqn:energy}  (see  \cite{Dafermos}, \cite{Evans} and \cite{Wloka}). Since we will have a similar situation to estimate the solution $w$ of \eqref{Semilinear for w}, the details about how to handle the mixed derivatives $\partial_t\partial_{x_j},\,1\le j\le n$ can be seen in the proof of Proposition \ref{Existence of solution to semi-linear equation}.
\end{proof}

	\noindent
	{\it Proof of Proposition \ref{Existence of solution to semi-linear equation}.} 
	First we prove the existence of a solution. 
	We simply write \eqref{Semilinear for w} as 
	\begin{equation}\label{eqn:semi-2}
	\begin{cases}
	L[w]=\epsilon\widetilde{\mathcal{F}}(t,x,\nabla_{t,x}w),
	\text{ $(t,x)\in Q_T$}, \\
	w(0,x)=\PD_t w(0,x)=0, \ \text{ $x\in\Omega$},\\
	w(t,x)=0, \ (t,x)\in\PD Q_{T}.
	\end{cases}
	\end{equation}
	In order to solve \eqref{eqn:semi-2}, we define a series of functions $\{w_j\}$ by
	\begin{align*}
	L[w_1] &= \epsilon \widetilde{\mathcal{F}}(t,x, 0), \\
	L[w_2] &= \epsilon \widetilde{\mathcal{F}}\lb t,x, {\nabla_{t,x}} w_1\rb, \\
	& \vdots \\
	L[w_j] &= \epsilon \widetilde{\mathcal{F}}\lb t,x, {\nabla_{t,x}} w_{j-1}\rb,
	\,\,\,j=2,3, \cdots .
	\end{align*}
	We first prove that $w_j\in X_m$ for each $j$ is bounded for any small enough $\epsilon>0$.
	By \eqref{eqn:energy} and 
	Lemma \ref{lem:composite}, if $\sup_{t\in[0,t]}[w_{j-1}(t)]_m\le M$,
	then we have
	\begin{align*}
	[w_j(t)]_{m}^2 &\le \epsilon^2 C \int\limits_{0}^{T} [\widetilde{\mathcal{F}}
	\lb t,x, {\nabla_{t,x}} w_{j-1}\rb]_{m-1}^2 \, dt \\
	& \le \epsilon^2 C \int\limits_{0}^{T} \lb 1+
 [w_{j-1}(t)]_m^{m-1}\rb^2\, dt,\,\,\,t\in[0,T]
	\end{align*}
	with some general constant $C>0$ which may differ by lines and may depend on $m$. 
	By \eqref{eqn:energy} and $$\widetilde{\mathcal{F}}(t,x,0)=\Big(1-\epsilon \Gamma_{00}(t,x,\nabla_{t,x}U;\epsilon)\Big)^{-1}F(x,\nabla_{{t},x} u_1,\nabla_{{t},x} u_2;\epsilon),
	$$
	we have $\sup_{t\in[0,T]}[w_1(t)]_m\le M$ if we take $\epsilon$ small enough. Then, we further take $\epsilon$ small enough if necessary, so that it satisfies
	\[ \epsilon \le \min\left\{\frac{1}{\sqrt{CT}}\frac{M}{1+ M^{m-1}},\, \frac12\right\}.\] Then it is easy to see by induction on $j\ge 2$ that 
	$$\displaystyle\sup_{t\in[0,T]}[w_j(t)]_m\le M,\,\,j\ge2.$$

	Next we prove that $\{w_j(t)\}_{j=1,2,\cdots }$ is a Cauchy sequence. Notice that
	\[ L[w_{j+1}-w_j]=\epsilon\Big\{\widetilde{\mathcal{F}}\lb t,x,{\nabla_{t,x}} w_j\rb
	- \widetilde{\mathcal{F}}(t,x,{\nabla_{t,x}} w_{j-1})\Big\},\]
	and $\{w_j\}$ is bounded. Then by \eqref{eqn:energy} and applying Lemma \ref{lem:composite} to 
	$$
	\begin{array}{ll}
	\widetilde{\mathcal{F}}(t,x,\nabla_{t,x}w_j)-\widetilde{\mathcal{F}}(t,x,\nabla_{t,x}w_{j-1})=\\
	\qquad\qquad\left\{\int\limits_{0}^{1}\nabla_q \widetilde{\mathcal{F}}(t,x,\nabla_{t,x}w_{j-1}+\theta\nabla_{t,x}(w_j-w_{j-1})\,d\theta\right\}\cdot\nabla_{t,x}(w_j-w_{j-1}),
	\end{array}
	$$
	we have 
	\[\Big[w_{j+1}(t)-w_j(t)\Big]_{m}^2 \le \epsilon^2 C \int\limits_0^t
	\Big[w_{j}({s})-w_{j-1}({s})\Big]_{m}^2  \, ds,\,\,t\in[0,T].
	\]
	By the choice of $\epsilon$, this immediately implies that $\{w_j(t)\}$ is a Cauchy sequence with respect to the norm $\sup_{t\in[0,T]}[\,\cdot\,]_m$.
	If we denote the limit of this Cauchy sequence by $w(t)$, then the standard regularity argument gives us that $w\in X_m$
	and it is a solution to
	\eqref{eqn:semi-2}.
	
	Next we prove the estimate \eqref{Estimate for w semi-linear}. 
	Differentiating \eqref{eqn:semi-2}, $m-1$ times with respect to 
	$t$ yields
	\begin{equation}\label{eqn:diffe-w}
	\overset{(m+1)}{w}(t)-A_U(t)\overset{(m-1)}{w}(t)=\sum_{k=1}^{m-1} 
	\begin{pmatrix}
	m-1 \\
	k
	\end{pmatrix}
	\overset{(k)}{A_U}(t)\overset{(m-1-k)}{w}(t)+\epsilon \PD_t^{m-1}
\widetilde{\mathcal{F}}. 
	\end{equation}
	By taking the $\Big\langle \cdot , \cdot \Big\rangle$ product of this 
	identity with $2\overset{(m)}{w}(t)$ and integrating by parts, 
	we have the identity 
	\begin{equation}\label{eqn:e-i}
	\begin{array}{ll}\Big\| \overset{(m)}{w}(t) \Big\|^2_{0} &- \Big\langle A_U(t)
	\overset{(m-1)}{w}(t), \overset{(m-1)}{w}(t)
	\Big\rangle =-\int\limits_0^t\Big\langle \dot{A}_U(\tau)\overset{(m-1)}{w}(\tau), \overset{(m-1)}{w}(\tau) \Big\rangle\, d\tau \\
	&+\int\limits_0^t {\bf A}(U(\tau); \overset{(m-1)}{w}(\tau), 
	\overset{(m)}{w}(\tau))\, d\tau 
	\\
	&+2\epsilon \int\limits_0^t \left\langle \PD_t^{m-1}
	\widetilde{\mathcal{F}}+ 
	\sum_{k=1}^{m-1}
	\begin{pmatrix}
	m-1 \\
	k
	\end{pmatrix}\overset{(k)}{A_U}(\tau)
	\overset{(m-1-k)}{w}(\tau),\overset{(m)}{w}(\tau) 
	\right\rangle \, d\tau,\,\,\,t\in[0,T].
	\end{array}
	\end{equation}
Here ${\bf A}(U(\tau);V(\tau),W(\tau))$ is defined by
$$
{\bf A}(U(\tau);V(\tau),W(\tau)):=\Big\langle
A_{U}(\tau)V(\tau),W(\tau)\Big\rangle-\Big\langle
A_{U}(\tau)W(\tau),V(\tau)\Big\rangle.
$$
and we have used the following identity obtained by integration by parts
\begin{equation}\label{integral identity}
\begin{array}{ll}2\int\limits_0^t\Big\langle A_U(\tau)
W(\tau),\dot W(\tau)\Big\rangle\,d\tau=\Big\langle A_U(\tau)W(t),W(t)\Big\rangle\\
\qquad\qquad-\int\limits_0^t\Big\langle \dot A_U(\tau))W(\tau),W(\tau)\Big\rangle\,d\tau
+\int\limits_0^t{\bf A}(U(\tau);W(\tau),\dot W(\tau))\,d\tau.
\end{array}
\end{equation}

Now we show the inequality
\begin{equation}\label{eqn:energy ineq}
\Big\| \overset{(m)}{w}(t)\Big\|_0^2+\Big\| \overset{(m-1)}{w}(t) \Big\|_1^2 \le
\epsilon^2 C+ K_{\epsilon}\int\limits_0^t \sum_{k=0}^{m}\Big\| \overset{(k)}{w}
(\tau)\Big\|_{m-k}^2 \, d\tau
\end{equation}
for any $t\in [0,T]$ with some general constant $C>0$ and a constant $K_\epsilon>0$ bounded with respect to $\epsilon$. To prove this we give the estimates for 
\newline
(i) $\int\limits_{0}^{t} {\bf A}(U(\tau);W(\tau),\dot W(\tau))\,d\tau$ with $W(\tau)=\overset{(m-1)}w(\tau)$,
\newline
(ii) a quadratic term $\lvert\nabla_{t,x} w(t)\rvert^2$ contained in ${\mathcal{F}}$.

We first deal with (i). Write $A_U(t)$ in the form
$$
A_U(t)\cdot=\widehat A_U(t)+\overrightarrow \ell\cdot\nabla_x
$$
with
$$
\widehat A_U(t):=(\overrightarrow M\cdot\nabla_x)\partial_t\cdot+\nabla_x\cdot(N\nabla_x\cdot),
$$
where $\overrightarrow M=\overrightarrow M(t,x,\nabla_{t,x}U;\epsilon)$ and $\overrightarrow\ell=\overrightarrow\ell(t,x,\nabla_{t,x}U;\epsilon)$ are real vectors and $N=N(t,x,\nabla_{t,x}U;\epsilon)$ is a positive matrix. Then by integrating by parts, we can have the estimate
\begin{equation}\label{estimate for bf A}
\bigg|\int\limits_{0}^{t}{\bf A}(U(\tau);W(\tau),\dot W(\tau))\,d\tau\bigg|\le C\left\{\big(\Vert\dot W(t)\Vert_0^2+\Vert W(t)\Vert_1^2\big)+
\int\limits_{0}^{t} \big(\Vert\dot W(\tau)\Vert_0^2+\Vert W(\tau)\Vert_1^2\big)\,d\tau\right\}
\end{equation}
with a general constant $C>0$. In fact by defining $\widehat{\bf A}(U(t);W(t),\dot W(t))$ similarly as ${\bf A}(U(t);W(t),\dot W(t))$, we have 
$$
\int\limits_{0}^{t}\widehat{\bf A}(U(\tau);W(\tau),\dot W(\tau))=J_1+J_2
$$
with
$$
\begin{array}{ll}
J_1&:=\int\limits_{0}^{t}\int\limits_{\Omega}\left\{\nabla_x\cdot(N\nabla_x W(\tau))\dot W(\tau)-\nabla_x\cdot(N\nabla_x \dot W(\tau))W(\tau)\right\}\,dx\,d\tau=0,\\
J_2&:=\int\limits_{0}^{t}\int\limits_{\Omega}\left\{(\overrightarrow M\cdot\nabla_x)\dot W(\tau))\dot W(\tau)-(\overrightarrow M\cdot\nabla_x)\ddot W(\tau))W(\tau)\right\}\,dx\,d\tau\\
&=\int\limits_{\Omega}\left\{(\nabla_x\cdot\overrightarrow M)W(t)\dot W(t)+\dot W(t)(\overrightarrow M\cdot\nabla_x)W(t)\right\}\,dx\\
&\quad-\int\limits_{0}^{t}\left\{(\nabla_x\cdot\overrightarrow M)W(\tau)\dot W(\tau)+(\nabla_x\cdot\partial_\tau\overrightarrow M)W(\tau)\dot W(\tau)\right\}\,dx\,d\tau.
\end{array}
$$
Further it is easy to see that $\big|\int\limits_{0}^{t}{\bf A}(U(\tau);W(\tau),\dot W(\tau))\,d\tau\big|$ coming from $\overrightarrow\ell\cdot\nabla_x$ can be absorbed into the second term of the right hand side of \eqref{estimate for bf A}. Hence taking these into account we can have \eqref{estimate for bf A}.

Next we deal with (ii). Let $\epsilon>0$ be small enough such that {$\| \partial_t w(t) \|_{m-1}<1$ and $\| w(t) \|_m <1$}. 
Then from the Sobolev embedding theorem that 
the quadratic term $\lvert\nabla_{t,x} w(t)\rvert^2$ contained in 
${\mathcal{F}}$ is estimated as follows. 	
	\begin{align*}
		\begin{aligned}
	\int\limits_{\Omega}{\Big\lvert\nabla_{t,x}} w(\tau)\Big\rvert^2 \Big|\overset{(m)}{w}(\tau)\Big|\,dx 
	& \leq \sup_{\Omega} | {\nabla_{t,x}} w(\tau) | \cdot \Big\| \overset{(m)}{w}(\tau)\Big\|_{0} 
	\cdot \Big\| {\nabla_{t,x}} w(\tau)\Big\|_{0} \\
	&{  \leq C \Big\| \nabla_{t,x} w(\tau) \Big\|_{m-1} \cdot 
	  \Big\| \overset{(m)}{w}(\tau) \Big\|_0 \cdot 
	  \Big\| \nabla_{t,x}w(\tau)\Big\|_0} \\
	& { \leq C \left\{ \Big\| \overset{(1)}{w}(\tau)\Big\|_{m-1}+  \Big\| w(\tau) \Big\|_{m}\right\} \cdot \Big\| \overset{(m)}{w}(\tau)\Big\|_{0} \cdot 
	\left\{ \Big\| \overset{(1)}{w}(\tau) \Big\|_0 + \Big\| \nabla_x w(\tau)\Big\|_0
	\right\} }\\
	& \leq C\lb\Big\| w(\tau) \Big\|_{m}^{2} 
	{+\Big\| \overset{(1)}{w}(\tau)                \Big\|_{m-1}^2}+\Big\| \overset{(m)}{w}(\tau)\Big\|^2_{0} \rb, 
	\end{aligned}
	\end{align*}
	for any $\tau \in [0,T]$. 
Consequently, 
using identity \eqref{eqn:e-i}, it follows from Lemma 
	\ref{elliptic part} 
	and a straightforward computation 
	(see, e.g., \cite[Theorem 3.1 pp. 274-277] {Dafermos}) 
	that we have 
	\eqref{eqn:energy ineq} for 
	sufficiently small $\epsilon >0$.
	
	To finish the proof we want to derive the estimate
	\begin{equation}\label{key estimate}
	\sum_{k=0}^m\Big\| \overset{(k)}{w}(t) \Big\|_{m-k}^2 \le
	\epsilon^2 C + K_{\epsilon} \int\limits_0^t 
	\sum_{k=0}^m\Big\|\overset{(k)}{w}(\tau)\Big\|_{m-k}^2 \, d\tau,\,\, t\in [0,T]
	\end{equation}
	with a general constant $C>0$ and a constant $K_\epsilon>0$ bounded with respect to $\epsilon$. Once we have this estimate, Gronwall's inequality allows us to prove estimate \eqref{Estimate for w semi-linear}, which implies that 
	solutions are unique. In order to see \eqref{key estimate}, we prove by induction on $\ell=0,1,\cdots,m-1$ the following estimate
	\begin{equation}\label{inductive estimate}
	\sum_{k=\ell}^m\Big\| \overset{(k)}{w}(t) \Big\|_{m-k}^2 \le
	\epsilon^2 C + K_{\epsilon} \int\limits_{0}^{t} 
	\sum_{k=0}^m\Big\|\overset{(k)}{w}(\tau)\Big\|_{m-k}^2 \, d\tau,\,\, t\in [0,T]
	\end{equation}
	with another general constant $C>0$ and another constant $K_\epsilon$. 
	By \eqref{eqn:energy ineq}, we have already proven \eqref{inductive estimate} for $\ell=m-1$. Assume \eqref{inductive estimate} holds for some $1\le\ell\le m-2$. Then we want to show that \eqref{inductive estimate} holds for $\ell-1$.
	We first have from \eqref{eqn:diffe-w}, the following identity
	\begin{equation}\label{eqn:w^l identity}
	\begin{array}{ll}
	-A_U(t)\overset{(\ell-1)}{w}(t) =  -\overset{(\ell+1)}{w}(t)+\\
	\qquad\qquad\qquad\displaystyle\sum_{i=1}^{\ell-1}\int\limits_{0}^{t}
	\begin{pmatrix}
	\ell-1 \\
	i
	\end{pmatrix}\left(\overset{(i)}{A_U}(\tau)\overset{(\ell-i)}{w}(\tau)+\overset{(i+1)}{A_U}(\tau)\overset{(\ell-i-1)}{w}(\tau)\right)\,d\tau
	+\epsilon \PD_t^{ l-1} {\widetilde{\mathcal{F}}}(t),\,\,t \in [0,T].
	\end{array}
	\end{equation}
	Next by using the coercivity of  $A_U$ given in Lemma \ref{elliptic part}, we have the following regularity estimate
	\begin{equation}\label{elliptic regularity estimate}
	\Vert z\Vert_{1+r}\lesssim\Vert z\Vert_1+\Vert g\Vert_{-1+r},\,\,r=0,1,\cdots
	\end{equation}
	for any solution $z\in W_0^{1,2}(\Omega)\cap W^{1+r,2}(\Omega)$ satisfying $-A_U(t)z=g\in W^{-1+r,2}(\Omega)$
	in $\Omega$, where the notation $\lesssim$ stands for $\le$ modulo multiplication by a positive general constant (see Chapter 20, (114) in \cite{Wloka}).
	By using \eqref{elliptic regularity estimate} with $r=m-\ell$ and
	$$\overset{(\alpha)}{w}(t)=\int\limits_{0}^{t}\frac{ (t-\tau)^{s-1}}{(s-1)!}\overset{(\alpha+s)}{w}(\tau)\,d\tau,\,\,
	\alpha,s\in{\mathbb Z},\,\alpha\ge0,\,s\ge1,
	$$ 
	we have from \eqref{eqn:w^l identity} the following estimate
	\begin{equation}\label{m=l to m=l-1 estimate}
	\left\{
	\begin{array}{ll}
	\Vert\overset{(\ell-1)}{w}(t)\Vert_{m-(\ell-1)}^2\lesssim\\
	\quad\Vert\overset{(\ell+1)}{w}\Vert_{m-(\ell+1)}^2+
	\Vert\int\limits_{0}^{t}\frac{(t-\tau)^{m-\ell-1}}{(m-\ell-1)!}\overset{(m-1)}{w}(\tau)\,d\tau\Vert_1^2
	+\Vert\big((\overrightarrow M\cdot\nabla_x)\partial_t+\overrightarrow\ell\cdot\nabla_x\big)\overset{(\ell-1)}w(t)\Vert_{m-(\ell+1)}^2
	\\
	\quad+\displaystyle\sum_{i=1}^{\ell-1}\left\{
	\left(\int\limits_{0}^{t}\frac{(t-\tau)^{i-1}}{(i-1)!}\Vert
	\overset{(\ell-1)} {w}(\tau)\Vert_{m-(\ell-1)}\,d\tau\right)^2+\left(\int\limits_{0}^{t}\frac{(t-\tau)^i}{i!}\Vert\overset{(\ell-1)} {w}(\tau)\Vert_{m-(\ell-1)}\,d\tau\right)^2\right\}+
	\epsilon^2 C.
	\end{array}
	\right.
	\end{equation}
	Here note that
	$$
	\begin{array}{ll}
	\Vert\int\limits_{0}^{t}\frac{(t-\tau)^{m-\ell-1}}{(m-\ell-1)!}\overset{(m-1)}{w}(\tau)\,d\tau\Vert_1^2
	\le \Big(\int\limits_{0}^{t}\frac{(t-\tau)^{m-\ell-1}}{(m-\ell-1)!}\Vert\overset{(m-1)}{w}(\tau)\Vert_1\,d\tau\Big)^2,\\
	\Vert\big((\overrightarrow M\cdot\nabla_x)\partial_t+\overrightarrow\ell\cdot\nabla_x\big)\overset{(\ell-1)}w(t)\Vert_{m-(\ell+1)}^2
	\lesssim \Vert\overset{(\ell)}w(t)\Vert_{m-\ell}^2+
	\left(\int\limits_{0}^{t}\Vert\overset{(\ell)}w(\tau)\Vert_{m-\ell}\,d\tau\right)^2.
	\end{array}
	$$
	Then together with \eqref{inductive estimate} for {$k=\ell+1$ in its right hand side} and a direct computation, we have \eqref{inductive estimate} for $m=\ell-1$. Thus Proposition \ref{Existence of solution to semi-linear equation} is proved. 
\subsection{Proof of Theorem \ref{Existence of w epsilon}.}${}$

Using Proposition \ref{Existence of solution to semi-linear equation}, we have for any small enough $\epsilon>0$, there exists a unique solution $w\in Z(M)$ to \eqref{Semilinear for w}. Thus, the map $T:Z(M)\rightarrow Z(M)$ given by $T(U)=w$ is well-defined, where $w$ is the solution to \eqref{Semilinear for w}. Now the idea is to use the fixed point  argument to prove that for any $\epsilon>0$ small enough, there exists a unique solution $w$ to the initial boundary value problem \eqref{equation for w in final form}. More precisely we will prove that $T:Z(M)\rightarrow Z(M)$ is a contraction mapping. To begin with let $T(U_{i})=w_{i}$ for $i=1,2$, where $w_{i}$ is the solution to semi-linear wave equation \eqref{Semilinear for w} for $U=U_{i}$. Let $W=w_{1}-w_{2}$ and  $V=U_{1}-U_{2}$, then $W$ will satisfy the following initial boundary value problem
\begin{equation}\label{Equation for W non-linear}
\begin{aligned}
\begin{cases}
\partial_{t}^{2}W(t)-A_{U_{1}}(t)W(t)=\left\{A_{U_{1}}(t)-A_{U_{2}}(t)\right\}w_{2}(t) +\epsilon\Big( 1-\epsilon\Gamma_{00}(t,x,\nabla_{t,x}U,\epsilon)\Big)^{-1}\overrightarrow{G}(t,x,\nabla_{t,x}w_{1};\epsilon)\cdot\nabla_{t, x}W(t),\\
\qquad +\epsilon\Big( 1-\epsilon\Gamma_{00}(t,x,\nabla_{t,x}U,\epsilon)\Big)^{-1}\left\{\overrightarrow{G}(t,x,\nabla_{t,x}w_{1};\epsilon)- \overrightarrow{G}(t,x,\nabla_{t,x}w_{2};\epsilon)\right\}\cdot\nabla_{t,x}w_{2}(t),\ (t,x)\in Q_{T},\\
W(0,x)=\partial_{t}W(0,x)=0,\ x\in\Omega,\\
W(t,x)=0,\ (t,x)\in\partial Q_{T}.
\end{cases}
\end{aligned}
\end{equation}
Here we have suppressed $x$ variable if it is clear from the context. Now let $\overrightarrow{\mathcal{G}}:=\Big( 1-\epsilon\Gamma_{00}(t,x,\nabla_{t,x}U,\epsilon)\Big)^{-1}\overrightarrow{G}$. Multiply \eqref{Equation for W non-linear} by $2\partial_{t}W$ and integrate over $[0,t]\times \Omega$, we have
\begin{align*}
\begin{aligned}
&||\dot W(t)||^{2}_{0}-\Big\langle A_{U_{1}}(t)W(t),W(t)\Big\rangle=-\int\limits_{0}^{t}\Big\langle \left(\dot A_{U_{1}}(\tau))\right)W(\tau),W(\tau)\Big\rangle d\tau+\int\limits_0^t{\bf A}(U_{1}(\tau);W(\tau),\dot W(\tau))\,d\tau\\
& \qquad +2\int\limits_{0}^{t}\Big\langle \Big[A_{U_{1}}(\tau)-A_{U_{2}}(\tau)\Big]w_{2}(\tau),\dot W(\tau)\Big\rangle d\tau +2\epsilon \int\limits_{0}^{t}\Big\langle \overrightarrow{\mathcal{G}}(t,x,\nabla_{t,x}w_{1};\epsilon)\cdot\nabla_{t,x}W(\tau),\dot W(\tau)\Big\rangle d\tau\\
& \qquad +2\epsilon \int\limits_{0}^{t}\Big\langle\left\{\overrightarrow{\mathcal{G}}({t},x,\nabla_{t,x}w_{1};\epsilon)- \overrightarrow{\mathcal{G}}({t},x,\nabla_{t,x}w_{2};\epsilon)\right\}\cdot\nabla_{t,x}w_{2}(\tau),\dot W(\tau)\Big\rangle d\tau,
\end{aligned}
\end{align*}
where $\dot A_{U_{1}}(\tau):=\partial_\tau A_{U_{1}}(\tau)$ and
$\dot W(\tau):=\partial_\tau W(\tau)$.

Using the expression  $\overrightarrow{G}$ given in \eqref{Notations}, we have 
\begin{equation*}
\begin{aligned}
\epsilon\lb \overrightarrow{G}(t,x,\nabla_{t,x}w_{1};\epsilon)-\overrightarrow{G}(t,x,\nabla_{t,x}w_{2};\epsilon)\rb 
=\epsilon\lb(\nabla_{t,x}\cdot K)(t,x,\epsilon\nabla_{t,x}w_{1};\epsilon)-(\nabla_{t,x}\cdot K)(t,x,\epsilon\nabla_{t,x}w_{2};\epsilon)\rb.
\end{aligned}
\end{equation*}
Hence by using \eqref{estimate of R vector}, the coercivity \eqref{coercivity} 
and estimate \eqref{estimate on A_U} 
given in Lemma \ref{elliptic part}, we have 
\begin{align*}
\begin{aligned}
\lVert\dot W(t)\rVert^{2}_{0}+\lVert W(t)\rVert_{1}^{2}&\leq \epsilon^2  C\sup_{t\in [0,T]}{ \left\{ \lVert \partial_t V(t) \rVert_0^2+ \lVert V(t)\rVert_{1}^2\right\}} +K_{\epsilon}\int\limits_{0}^{t}\Big( \lVert\dot W(\tau)\rVert^{2}_{0}+\lVert W(\tau)\rVert_{1}^{2}\Big)d\tau 
\end{aligned}
\end{align*}
with a constant $K_\epsilon>0$ bounded with respect to $\epsilon$ and a general constant $C>0$.

Now we equip $Z(M)$ with the metric $\rho$ defined by 
\begin{align*}
\rho(f,g):=\max_{t\in[0,T]}\left\{||f(t)-g(t)||_{1}^{2}+||\partial_{t}f(t)-\partial_{t}g(t)||^{2}_{0}\right\}^{\frac{1}{2}}.
\end{align*}
Finally using  Grownwall's inequality, we have 
\begin{align*}
\rho\lb T(U_{1}),T(U_{2})\rb\leq \epsilon Ce^{K_{\epsilon}T}\rho(U_1,U_2)
\end{align*}
with another general constant $C>0$. This implies that $T$ is a contraction mapping for $\epsilon>0$ small enough. Therefore, for each $\epsilon>0$ small enough, there exists $w\in Z(M)$ such that $T(w)=w$ and it will satisfies the estimate  \eqref{estimate of w } which follows from \eqref{Estimate for w semi-linear}. 
Hence, Theorem \ref{Existence of w epsilon} is proved.

Now Theorem \ref{Unique solvability of direct problem non-linear} follows from Theorem \ref{Existence of w epsilon}.

\noindent

\subsection{Analysis of input-output map in $\epsilon$-expansion}\label{DN map in Epsilon expansion}${}$

Let $\Lambda_{a}$ denote the input-output map corresponding to \eqref{equation for u1 non-linear}. Using the $\epsilon$-expansion  of the solution to \eqref{equation of interest} we  show that $\Lambda_{a}$ can be reconstructed from $\Lambda_{\overrightarrow{C},a}$. In particular, we prove the following lemma. 

\begin{lemma}\label{Reconstruting the DN map for linear equation}
For $m\geq [n/2]+3$ and $(\phi,\psi,f)\in B_{M}$ satisfying the compatibility conditions of order $m-1$, we have 
\begin{align}\label{Difference of DN maps divided by epsilon}
\begin{aligned}
\lim_{\epsilon\to 0}\frac{1}{\epsilon}\Big\|\Lambda_{\overrightarrow{C},a}(\epsilon \phi,\epsilon \psi, \epsilon f)-\epsilon \Lambda_{a}(\phi,\psi,f)\Big\|_{\widetilde X_m}=0,
\end{aligned}
\end{align}
and 
\begin{align}\label{Difference of DN maps divided by epsilon2}
\lim_{\epsilon\to 0}\frac{1}{\epsilon^{2}}\Big\|\Lambda_{\overrightarrow{C},a}(\epsilon \phi,\epsilon \psi, \epsilon f)-\epsilon \Lambda_{a}(\phi,\psi,f)-\epsilon^2\lb\PD_{\nu}u_2+\lb 0,\nu(x)\rb\cdot\vec{b}(t,x)|\nabla_{t,x}u_1(t,x)|^2\rb\Big\|_{\widetilde X_m}=0,
\end{align}
where ${\widetilde X_m}$ is the one defined in Remark \ref{solvablility with nonzero D-data}, (i).
\begin{proof}
	Using the $\epsilon$-expansion of the solution $u$ given in Theorem \ref{Unique solvability of direct problem non-linear}, we have the $\epsilon$-expansion of $\Lambda_{\overrightarrow{C},a}(\epsilon \phi,\epsilon \psi, \epsilon f)$ is given by 
	\begin{align*}
	\begin{aligned}
	&\Lambda_{\overrightarrow{C},a}(\epsilon \phi,\epsilon \psi, \epsilon f)={ \lb \Big[\PD_{\nu}u^{\phi,\psi,f}+(0,\nu(x))\cdot\overrightarrow{{C}}(t,x,\nabla_{t,x}^{\phi,\psi,f}\Big]\Big|_{\PD Q_{T}},u^{\phi,\psi,f}|_{t=T},\PD_{t}u^{\phi,\psi,f}|_{t=T}\rb} \\
	&\ \ \ \ \ {=\epsilon\lb \PD_{\nu}u_{1},u_{1}|_{t=T},\PD_{t}u_{1}|_{t=T}\rb\big|_{\PD Q_{T}}+\epsilon^{2}\lb \PD_{\nu}u_{2}+(0,\nu(x))\cdot\vec{b}(t,x)\lvert\nabla_{t,x}u_{1}\rvert^{2}\rb\Big|_{\PD Q_{T}}+O(\epsilon^{3}),\ \lb \epsilon\to 0\rb.}
	\end{aligned}
	\end{align*}
	Using this we can have \eqref{Difference of DN maps divided by epsilon} and \eqref{Difference of DN maps divided by epsilon2}. 
\end{proof}
\end{lemma}

\section{Proof of Theorem $\ref{Uniqueness theorem}$}\label{Proof of uniqueness Theorem}
\setcounter{equation}{0}		
\renewcommand{\theequation}{3.\arabic{equation}}${}$
To prove the theorem, we will use  the $\epsilon$-expansion of the solution $u^{(i)}$ to $\eqref{equation for ui}$. Following Theorem \ref{Unique solvability of direct problem non-linear} we have the $\epsilon$-expansion of the solution $u^{(i)}$ to \eqref{equation for ui} is given by
\begin{align}\label{epsilon approximation of solution of u i}
u^{(i)}(t,x)=\epsilon u^{(i)}_{1}(t,x)+\epsilon^{2}u^{(i)}_{2}(t,x)+O(\epsilon^3).     \end{align}
By the straight forward calculations, we have
\begin{align*}
\left\{\begin{array}{ll}
\partial_{t}^{2}u^{(i)}=\epsilon\partial_{t}^{2}u^{(i)}_{1}+\epsilon^{2}\partial_{t}^{2}u^{(i)}_{2}+O(\epsilon^3),\qquad\qquad\qquad\qquad\qquad\qquad\qquad\qquad\\
a_{i}u^{(i)}=\epsilon a_{i}u_{1}^{(i)}+\epsilon^{2}a_{i}u_{2}^{(i)}+O(\epsilon^{3}),\\
\nabla_{{t},x}u^{(i)}=\epsilon \nabla_{t,x}u^{(i)}_{1}+\epsilon^{2}\nabla_{{t},x}u^{(i)}_{2}+O(\epsilon^3),\qquad\qquad\qquad\qquad\qquad\qquad\qquad\qquad\\
\Delta u^{(i)}={\epsilon \Delta u^{(i)}_{1}+\epsilon^{2}\Delta u^{(i)}_{2}+O(\epsilon^3)},\qquad\qquad\qquad\qquad\qquad\qquad\qquad\qquad\\
\overrightarrow{C}^{(i)}\lb {t},x,\nabla_{{t,}x}u^{(i)}\rb={\epsilon^{2} \lvert\nabla_{t,x}u^{(i)}_{1}\rvert^{2}\vec{b}^{(i)}+O(\epsilon^{3})},\\
\nabla_{{t,}x}\cdot\overrightarrow{C}^{(i)}\lb {t,}x,\nabla_{{t,}x}u^{(i)}\rb= \epsilon^{2}\nabla_{{t,}x}\cdot\lb \Big\lvert\nabla_{{t},x}u_{1}^{(i)}\Big\rvert^{2}\vec{b}\rb+O(\epsilon^{3}).
\end{array}
\right.
\end{align*}
Substitute $\eqref{epsilon approximation of solution of u i}$ into $\eqref{equation for ui}$, and arrange the terms into ascending order of power of $\epsilon$ by using the above calculations. Further setting the coefficients of $\epsilon$ and $\epsilon^2$ equal to zero. Then we have the following equations for $u_{1}^{(i)}$ and $u_{2}^{(i)}$:
\begin{align}\label{equation for u1 i}
\begin{aligned}
\begin{cases}
\partial_{t}^{2}u^{(i)}_{1}(t,x)-\Delta u^{(i)}_{1}(t,x)+a_{i}(x)u_{1}^{(i)}(t,x)=0, \ \   (t,x)\in Q_{T},\\
u^{(i)}_{1}(0,x)=\phi(x),\ \partial_{t}u^{(i)}_{1}(0,x)=\psi(x), \ \  x\in\Omega,\\
u_{1}^{(i)}(t,x)=f(t,x),\ \ (t,x)\in\partial Q_{T},
\end{cases}
\end{aligned}
\end{align}
\begin{align}\label{equation for u2 i}
\begin{aligned}
\begin{cases}
\partial_{t}^{2}u^{(i)}_{2}(t,x)-\Delta u^{(i)}_{2}(t,x)+a_{i}(x)u_{2}^{(i)}(t,x)={\nabla_{{t,}x}\cdot\lb \Big\lvert\nabla_{{t},x}u_{1}^{(i)}(t,x)\Big\rvert^{2}\vec{b}({t},x)\rb}, \  (t,x)\in Q_{T},\\
u^{(i)}_{2}(0,x)=\partial_{t}u^{(i)}_{2}(0,x)=0,\    x\in\Omega,\\
u_{2}^{(i)}(t,x)=0,  \ (t,x)\in \partial Q_{T}.
\end{cases}
\end{aligned}
\end{align}

\subsection{Proof of the uniqueness for $a$}${}$

By knowing  $\Lambda_{\overrightarrow{C}^{(i)},a_{i}}^T(\epsilon\phi,\epsilon\psi,\epsilon f),\,i=1,2$ for any $\lb \phi,\psi,f\rb \in B_M$, and $0<\epsilon<\epsilon_0$,  we do know $\Lambda_{a_{i}}^T,\,i=1,2$ (see  Lemma  \ref{Reconstruting the DN map for linear equation}) and $\Lambda_{\overrightarrow{C}^{(1)},a_{1}}=\Lambda_{\overrightarrow{C}^{(2)},a_{2}}$ gives $\Lambda_{a_{1}}=\Lambda_{a_{2}}$.
Therefore, using the arguments from \cite{Rakesh_Reconstruction} we can reconstruct $a_{i}(x)$ from $\Lambda_{a_{i}}$, and  from \cite{Rakesh_Symes_Uniqueness}, we have $a_{1}=a_{2}$ in $\Omega$. 
We denote this common $a_{i},\,i=1,2$ by $a$, i.e.
\begin{equation}\label{gamma}
a=a_1=a_2\,\,\,\text{in\, $\Omega$}.
\end{equation}
Before closing this subsection, we give some by products of \eqref{gamma}.	Since the given  data $(\phi,\psi,f)$ is the same for $u_1^{(i)},\,i=1,2$, therefore we do know $u_1^{(1)}=u_1^{(2)}\,\,\text{in $Q_T$} $ and we denote this common solution by $u_1=u_1^{\phi,\psi,f}$, i.e.
\begin{equation}\label{u_1}
u_1=u_1^{\phi,\psi,f}= u_1^{(1)}=u_1^{(2)}\,\,\text{in $Q_{T}$}.
\end{equation}

\subsection{Proof of the uniqueness for $\vec{b}({t,}x)$}${}$

Now we abuse the notations to denote $\vec{b}({t,}x):=\vec{b}^{(1)}({t,}x)-\vec{b}^{(2)}({t,}x)$ so that $\overrightarrow{P}({t,}x,q):=\overrightarrow{P}^{(1)}({t,}x,q)-\overrightarrow{P}^{(2)}({t,}x,q)=\lvert q\rvert^{2}\vec{b}({t,}x)$. Also we denote the solutions of $\eqref{equation for u2 i}$   by $u_2^{(i)\phi,\psi,f},\,i=1,2$ with $u_1^{(i)}=u_1^{\phi,\psi,f},\,i=1,2$ and define $u^{\phi,\psi,f}_{2}(t,x):=u_{2}^{(1)\phi,\psi,f}(t,x)-u_{2}^{(2)\phi,\psi,f}(t,x)$. Then, from $\eqref{equation for u1 i}$ and $\eqref{equation for u2 i}$, $u_1(t,x):=u^{\phi,\psi,f}_{1}(t,x)\in X_m$ and $u_2(t,x):=u^{\phi,\psi,f}_{2}(t,x)\in X_m$ are the unique solutions to the following initial boundary value problems:
\begin{align}\label{equation for u 1 in uniqueness case}
\begin{aligned}
\begin{cases}
\partial_{t}^{2}u_{1}(t,x)-\Delta u_{1}(t,x)+a(x)u_{1}(t,x)=0, \   (t,x)\in Q_{T},\\
u_{1}(0,x)=\phi(x),\ \partial_{t}u_{1}(0,x)=\psi(x), \ x\in\Omega,\\
u_{1}(t,x)=f(t,x),\  (t,x)\in\partial Q_{T}
\end{cases}
\end{aligned}
\end{align}
and
\begin{align}\label{equation for u 2 uniqueness case}
\begin{aligned}
\begin{cases}
\partial_{t}^{2}u_{2}(t,x)-\Delta u_{2}(t,x)+a(x)u_{2}(t,x)={\nabla_{{t,}x}\cdot\lb \Big\lvert\nabla_{{t},x}u_{1}(t,x)\Big\rvert^{2}\vec{b}({t},x)\rb}, \ (t,x)\in Q_{T},\\
u_{2}(0,x)=\partial_{t}u_{2}(0,x)=0,\   x\in\Omega,\\
u_{2}(t,x)=0,\ (t,x)\in \partial Q_{T}
\end{cases}
\end{aligned}
\end{align}
respectively.

From \eqref{Difference of DN maps divided by epsilon2}
and Lemma \ref{Reconstruting the DN map for linear equation}, we have that 
\begin{align}\label{equality of Neumann data at epsilon2 level}
{u_{2}|_{t=T}=\PD_{t}u_{2}|_{t=T}=\left[\partial_{\nu}u_{2}(t,x)
+(0,\nu(x))\cdot\vec{b}(t,x)\lvert\nabla_{t,x}u_{1}^{\phi,\psi,f}(t,x)\rvert^{2}\right]\Big|_{\partial Q_{T}}=0,}
\end{align}
where $\partial_\nu u_2$ is the Neumann derivative of $u_2$ given by $\partial_\nu u_2=\nu\cdot\nabla_{x} u_2$ {and $\nu(x)$ stands for the outward unit normal to $\PD\Omega$ at $x\in\PD\Omega$.} Now let $w$ be any solution to the following equation
\begin{align}\label{Equation for w}
\begin{aligned}
\PD_{t}^{2}w(t,x)-\Delta w(t,x)+a(x)w(t,x)=0,\ (t,x)\in Q_{T}.
\end{aligned}
\end{align}
Multiplying \eqref{equation for u 2 uniqueness case} by $w$ and integrating over $Q_{T}$, we have 
\begin{align*}
\begin{aligned}
\int\limits_{Q_{T}}\lb 	\partial_{t}^{2}u_{2}(t,x)-\Delta u_{2}(t,x)+a(x)u_{2}(t,x)\rb w(t,x)dxdt=\int\limits_{Q_{T}}{\nabla_{{t,}x}\cdot\lb \Big\lvert\nabla_{{t},x}u_{1}(t,x)\Big\rvert^{2}\vec{b}({t},x)\rb} w(t,x)dxdt. 
\end{aligned}
\end{align*}
Now using the integration by parts and using \eqref{equality of Neumann data at epsilon2 level}, we have 
\begin{align}\label{Identity with phi psi f}
\int\limits_{Q_{T}}\vec{b}({t},x)\cdot \nabla_{t,x}w({t},x)\lvert \nabla_{{t},x}u_{1}(t,x)\rvert^{2}dxdt=0\
\end{align}
holds for all solutions $u_{1}$ of \eqref{equation for u 1 in uniqueness case} and solutions $w$ of \eqref{Equation for w}.
We remark here that $w$ only needs to satisfy \eqref{Equation for w} is the advantage coming from taking the input-output map as our measurement.

Now let $u_{1}^{\phi_{1}\pm\phi_{2},\psi_{1}\pm\psi_{2},f_{1}\pm f_{2}}$ be solutions to \eqref{equation for u 1 in uniqueness case} when $\phi=\phi_{1}\pm\phi_{2}, \psi=\psi_{1}\pm\psi_{2}$ and $f=f_{1}\pm f_{2}$, respectively. Use the two sets of solution $u^{\phi,\psi,f}_{1}= u_{1}^{\phi_{1}\pm\phi_{2},\psi_{1}\pm\psi_{2},f_{1}\pm f_{2}}$ in \eqref{Identity with phi psi f} and subtract the two sets of equations. Then we have  
\begin{align}\label{Integral identity before GO}
\begin{aligned}
\int\limits_{\Rb^{1+n}}\beta_{w}(t,x)\nabla_{t,x}u_{1}^{\phi_{1},\psi_{1},f_{1}}\cdot\nabla_{t,x}u_{1}^{\phi_{2},\psi_{2},f_{2}}(t,x)dxdt=0, \ \lb \phi_{j},\psi_{j},f_{j}\rb\in B_{M}, \ j=1,2,
\end{aligned}
\end{align}
where $\beta_{w}(t,x)=\chi_{Q_{T}}\vec{b}({t},x)\cdot\nabla_{{t},x}w(t,x)$ with the characteristic function $\chi_{Q_{T}}$ of $Q_{T}$. In deriving the above identity, we have used the fact that $u_{1}^{\phi_{1}\pm\phi_{2},\psi_{1}\pm\psi_{2},f_{1}\pm f_{2}}=u_{1}^{\phi_{1},\psi_{1},f_{1}}\pm u_{1}^{\phi_{2},\psi_{2},f_{2}}$. 

Since the principle term in \eqref{equation for w in final form} has the coefficients which contains the functions involving the solution $u_{1}$ to \eqref{equation for u1 non-linear}, therefore to make the coefficients to be real-valued, we use the the real-valued semi-classical solutions for $u_{1}^{\phi_i,\psi_i,f_i}$, $i=1,2$ in \eqref{Integral identity before GO}.  Now from \cite{Kian damping,KV}, we can have the real-valued semi-classical solutions $u_{1}^{\phi_i,\psi_i,f_i}$, $i=1,2$ given as
\begin{align*}
\begin{aligned}
&u_{1}^{\phi_{1},\psi_{1},f_{1}}=e^{- \lb t+x\cdot \omega\rb /h}\lb \varphi(x+t\omega)+hR_{1}(t,x)\rb,\\
& u_{1}^{\phi_{2},\psi_{2},f_{2}}=e^{\lb t+x\cdot\omega\rb/h}\lb \varphi(x+t\omega)+hR_{2}(t,x)\rb,
\end{aligned}
\end{align*}
where $\omega\in \Sb^{n-1}$, $0<h\leq h_{0}$, $\varphi\in C_{0}^{\infty}(\Rb^{n})$ and $R_{i}(t,x)=R_{i}(t,x;h)$, $i=1,2$ satisfy the estimate
\begin{equation}\label{Extimate for correction term}
\begin{array}{ll}
\lVert R_{i}\rVert_{L^{2}(\overline{Q}_T)}+\lVert h\nabla_{t,x} R_{i}\rVert_{L^{2}(\overline{Q}_T)}\leq C,\,\,\, i=1,2,\,\,0<h\leq h_{0}
\end{array}
\end{equation}
here the constant $C>0$ depends only on $\Omega$, $T$, $a$ and we have suppressed $h$ for each $R_i(t,x)$ for simplicity. 
Using these choices for $u_{1}^{\phi_{1},\psi_{1},f_{1}}$ and $u_{1}^{\phi_{2},\psi_{2},f_{2}}$ in  \eqref{Integral identity before GO}, we have 
\begin{align*}
    \begin{aligned}
&-\frac{2}{h^{2}}\int\limits_{\Rb^{1+n}}\beta_{w}\varphi^{2}(x+t\omega)dxdt+\frac{1}{h}\int\limits_{\Rb^{1+n}}\beta_{w}\varphi(x+t\omega)(1,\omega)\cdot\lb h\nabla_{t,x}R_{1}-h\nabla_{t,x}R_{2}\rb dxdt\\
&\ -\frac{2}{h}\int\limits_{\Rb^{1+n}}\beta_{w}\varphi(x+t\omega)\lb R_{1}+R_{2}\rb dxdt+\int\limits_{\Rb^{1+n}}\beta_{w}\lb \lvert\nabla_{t,x}\varphi\rvert^{2}-2R_{1}R_{2}+h\nabla_{t,x}\varphi\cdot\lb \nabla_{t,x}R_{1}+\nabla_{t,x}R_{2}\rb\rb dxdt\\
& \ \ +\int\limits_{\Rb^{1+n}}\beta_{w}(1,\omega)\cdot\lb R_{2}\nabla_{t,x}\varphi -R_{1}\nabla_{t,x}\varphi+\varphi\nabla_{t,x}R_{1}-\phi\nabla_{t,x}R_{2}-hR_{1}\nabla_{t,x}R_{2}+hR_{2}\nabla_{t,x}R_{1}\rb dxdt=0,
\end{aligned}
\end{align*}
for any solution $w$ of \eqref{Equation for w}, $\omega\in \Sb^{n-1}$ and $\varphi\in C_{0}^{\infty}(\Rb^{n})$. 
Now multiplying by $h^{2}$ and taking $h\rightarrow 0$, we get 
\begin{align*}
\begin{aligned}
\int\limits_{\Rb^{1+n}}\beta_{w}(t,x)\varphi^{2}(x+t\omega)dxdt=0, \ \mbox{for all} \ \omega\in\Sb^{n-1}\ \mbox{and for all}\ \varphi\in C_{0}^{\infty}(\Rb^{n}). 
\end{aligned}
\end{align*}
After substituting $x+t\omega=y$, we get 
\begin{align*}
\int\limits_{\Rb^{n}}\int\limits_{\Rb}\varphi^{2}(y)\beta_{w}(t,y-t\omega)dt \D y=0,\,\,\ \varphi\in C_{0}^{\infty}(\Rb^{n}),\,\, \omega\in \Sb^{n-1}.
\end{align*}
Thus finally we have
\begin{align}\label{Light ray transform}
\int\limits_{\Rb}\beta_{w}(t,y-t\omega)dt=0,\,\,(t,y)\in\Rb^{1+n},\,\,\omega\in \Sb^{n-1}.
\end{align}
For each $\omega\in{\mathbb S}^{n-1}$, take $(r,y)\in{\mathbb R}\times{\mathbb R}^n$ such that $2r+y\cdot\omega=0$. Then $\ell:=(r,y+r\omega)\in (1,\omega)^\perp$. Hence by the change of variable $t=r+s$, we have
\begin{equation}\label{light ray transform}
\int_{\mathbb R}\beta_w(\ell+s(1,\omega))\,ds=0,\,\,\ell\in (1,\omega)^\perp,\,\omega\in{\mathbb S}^{n-1}.
\end{equation}

Based on this we will prove $\beta_{w}(t,y)=0$ in $\Rb^{1+n}$ by using the Fourier-slice theorem (see for example in \cite{Stefanov_Support_Theorem_Lorentzian_Manifold_2017}). We start by considering
\begin{align*}
\begin{aligned}
\widehat{\beta_{w}}(\zeta):=\int\limits_{\Rb^{1+n}}e^{i\zeta\cdot(t,x)}\beta_{w}(t,x)\,dxdt.
\end{aligned}
\end{align*}
Using the decomposition, 
$\mathbb{R}^{1+n}=\mathbb{R}(1,\omega)\oplus \ell$ and Fubini's theorem, we have
\begin{align*}
\wh{\beta_{w}}(\zeta)=\sqrt{2}\int\limits_{(1,\omega)^{\perp}}\int\limits_{\mathbb{R}}\beta_{w}(\ell+s(1,\omega))e^{-\I (\ell +s(1,\omega))\cdot\zeta}\, \D s\, \D \ell.
\end{align*}
By \eqref{Light ray transform} and $\zeta\in (1,\omega)^{\perp}$ implies 
\begin{align*}
\widehat{\beta_{w}}(\zeta)=\sqrt{2}\int\limits_{(1,\omega)^{\perp}}\int\limits_{\mathbb{R}}\beta_{w}(s(1,\omega)+\ell)e^{-i\ell\cdot\zeta}\, \D s\, \D\ell=0.
\end{align*}
Hence $\widehat{\beta_{w}}(\zeta)=0$ for all $\zeta\in\lb 1,\omega\rb^{\perp}$ and  $\omega\in\Sb^{n-1}$. Now since $\cup_{\omega\in\Sb^{n-1}}\lb 1,\omega\rb^{\perp}=\{(t,x): \lvert t\rvert \leq\lvert x\rvert\}$, we have $\widehat{\beta_{w}}(\zeta)=0$ for all space-like vectors $\zeta$, hence using the Paley-wiener theorem, we have $\widehat{\beta_{w}}(\zeta)=0$ for all $\zeta\in\Rb^{1+n}$. Thus we have $\beta_{w}(t,x)=0$ for all $\lb t,x\rb \in\Rb^{1+n}$ and $w$ solutions to \eqref{Equation for w} which gives us  $\vec{b}({t},x)\cdot\nabla_{{t},x}w(t,x)=0$ in $Q_{T}$ for all solution $w$ of \eqref{Equation for w}. { Now to prove that $\vec{b}(t,x)=0$ in $Q_{T}$, we use the following lemma. 
\begin{lemma}\label{Lemma about LI of solutions}
Suppose $n\geq 2$ and $N> \frac{1+n}{2}+2$. There exists solutions $v_{j}\in H^{N}(Q_{T})$, $0\leq j\leq n$ such that \[\mbox{det}\lb \frac{\PD v_{j}}{\PD x_{i}}\rb_{0\leq i,j\leq n}\neq 0\,\,\mbox{\rm a.e. in $Q_{T}$}\]
\begin{proof}
Let us choose $\omega_{j}\in \Sb^{n-1}$ for $0\leq j\leq n$ such that $(1,\omega_{0}),(1,\omega_{1}),(1,\omega_{2}),\cdots,(1,\omega_{n})$ are linearly independent. This can be done for example we can choose $\omega_{0}=\frac{1}{\sqrt{n}}(1,1,\cdots,1)$ and  $\omega_{j}=e_{j}$ for $1\leq j\leq n$ where $e_{j}$ represent the standard basis of $\Rb^{n}$.  
Then it is easy to see that $(1,\omega_{0}),(1,\omega_{1}),\cdots, (1,\omega_{n})$ are linearly independent. Next extending $a(x)$ to a function in $C_0^\infty({\mathbb R}^n)$, we choose the WKB solutions $v_{j}(t,x)$ for $0\leq j\leq n$ of $Lw:=\PD_{t}^{2}w(t,x)-\Delta w(t,x)+a(x)w(t,x)=0$ in ${\mathbb R}^{1+n}$ which take the following form
\begin{align}\label{vj expressions} 
v_{j}(t,x)=e^{i\lambda\lb t+x\cdot\omega_{j}\rb}\sum_{k=0}^{N}\frac{A_{jk}(t,x)}{\lb 2i\lambda\rb^{k}}+R_{j}(t,x)\,\, \mbox{with} \ N>\frac{1+n}{2}+2,\,\,\lambda\gg1
\end{align}
(see for example \cite{Salazar}). 
Observe that 
\begin{align*}
    \begin{aligned}
Lv_{j}=e^{i\lambda\lb t+x\cdot\omega\rb}\left[\lb2i\lambda \mathcal{L}+L\rb\lb A_{j0}(t,x)+\frac{A_{j1}(t,x)}{2i\lambda}+\frac{A_{j2}(t,x)}{\lb2i\lambda\rb^{2}}+\cdots+\frac{A_{jN}(t,x)}{\lb2i\lambda\rb^{N}}+e^{-i\lambda\lb t+x\cdot\omega\rb}R_{j}(t,x)\rb\right]
    \end{aligned}
\end{align*}
where $\mathcal{L}:=\PD_{t}-\omega\cdot\nabla_{x}$ is the transport operator. By equating the terms with same power of $2i\lambda$, we have 
\begin{align*}
    \begin{aligned}
&2i\lambda \mathcal{L}A_{j0}+\lb \mathcal{L}A_{j1}+LA_{j0}\rb+\frac{1}{2i\lambda}\lb \mathcal{L}A_{j2}+LA_{j1}\rb +\cdots+\frac{1}{\lb2i\lambda\rb^{N-1}}\lb \mathcal{L}A_{jN}+LA_{j,N-1}\rb \\
& \ \ \ \ \ \ \ \ +\frac{1}{\lb 2i\lambda\rb^{N}}LA_{jN}+e^{-i\lambda\lb t+x\cdot\omega_{j}\rb}LR_{j}=0.
    \end{aligned}
\end{align*}
Then we have the transport equations for $A_{jk},\,\,0\leq k\leq N$ given as
\begin{align}\label{Transport equations for Aj0}
    \mathcal{L}A_{j0}=0
    \end{align}
and for $1\leq k\leq N$ 
    \begin{align}\label{Transport equations for Ajk}
        \begin{aligned}
     \begin{cases}
    \mathcal{L}A_{jk}=-LA_{j,k-1},\\
    A_{jk}(0,x)=0.
    \end{cases}
        \end{aligned}
    \end{align}

We take $A_{j0}=1$ for the first equation in \eqref{Transport equations for Aj0}. After finding $A_{jk}$ for $0\leq k\leq N$, we take $R_{j}$ as the solution to 
\begin{align*}
    \begin{aligned}
    \begin{cases}
LR_{j}(t,x)=-e^{i\lambda\lb t+x\cdot\omega\rb}\frac{1}{\lb2i\lambda\rb^{N}}LA_{jN}(t,x)\,\,\mbox{for}\   (t,x)\in\Rb^{1+n}\\
R_{j}(t,x)=\PD_{t}R_{j}(t,x)=0\,\,\text{at $t=0$}.
\end{cases}
    \end{aligned}
\end{align*}
Now solving this Cauchy problem for $R_{j}$, we get that $R_{j}\in H^{N}(\Rb^{1+n})$. Hence restricting these solutions to $Q_{T}$ and using the Sobolev embedding theorem, we have $R_{j}$ will satisfy the following estimate $\lVert \nabla_{t,x}R_{j}\rVert_{L^{\infty}(Q_{T})}\leq C$ for some constant $C$ independent of $\lambda$.

Now consider the matrix 
\begin{align*}
\begin{aligned}
A(t,x,\lambda)&:=\lb \lb\frac{\PD v_{j}}{\PD x_{i}}\rb\rb_{0\leq i,j\leq n}\\&=
\begin{bmatrix}
i\lambda e^{i\lambda\lb t+x\cdot\omega_{0}\rb}+\PD_{t}\widetilde{R}_{0}&i\lambda w_{01}e^{i\lambda\lb t+x\cdot\omega_{0}\rb}+\PD_{1}\widetilde{R}_{0}&\cdots&i\lambda w_{0n}e^{i\lambda\lb t+x\cdot\omega_{0}\rb}+\PD_{n}\widetilde{R}_{0}\\
i\lambda e^{i\lambda\lb t+x\cdot\omega_{1}\rb}+\PD_{t}\widetilde{R}_{1}&i\lambda w_{11}e^{i\lambda\lb t+x\cdot\omega_{1}\rb}+\PD_{1}\widetilde{R}_{1}&\cdots&i\lambda w_{1n}e^{i\lambda\lb t+x\cdot\omega_{1}\rb}+\PD_{n}\widetilde{R}_{1}\\
\vdots&\vdots&\cdots&\vdots\\
i\lambda e^{i\lambda\lb t+x\cdot\omega_{n}\rb}+\PD_{t}\widetilde{R}_{n}&i\lambda w_{n1}e^{i\lambda\lb t+x\cdot\omega_{n}\rb}+\PD_{1}\widetilde{R}_{n}&\cdots&i\lambda w_{nn}e^{i\lambda\lb t+x\cdot\omega_{n}\rb}+\PD_{n}\widetilde{R}_{n}
\end{bmatrix},
\end{aligned}
\end{align*}
where $\omega_{ij}$ denote the $j'th$ component in $\omega_{i}\in\Sb^{n-1}$ and $\widetilde{R}_{j}(t,x)=e^{i\lambda\lb t+x\cdot\omega\rb}\sum_{k=1}^{N}\frac{A_{jk}(t,x)}{\lb 2i\lambda\rb^{k}}+R_{j}(t,x)$. Let us denote by $\alpha_{j} :=e^{i\lambda\lb t+x\cdot\omega_{j}\rb}$ for $0\leq j\leq n$, then  matrix $A(t,x,\lambda)$ becomes
\begin{align}\label{Matrix A}
A(t,x,\lambda)=
\begin{bmatrix}
i\lambda \alpha_{0}+\PD_{t}\widetilde{R}_{0}&i\lambda \alpha_{0}\omega_{01}+\PD_{1}\widetilde{R}_{0}&\cdots&i\lambda \alpha_{0}\omega_{0n}+\PD_{n}\widetilde{R}_{0}\\
i\lambda \alpha_{1}+\PD_{t}\widetilde{R}_{1}&i\lambda \alpha_{1}\omega_{11}+\PD_{1}\widetilde{R}_{1}&\cdots&i\lambda \alpha_{1}\omega_{1n}+\PD_{n}\widetilde{R}_{1}\\
\vdots&\vdots&\cdots&\vdots\\
i\lambda \alpha_{n}+\PD_{t}\widetilde{R}_{n}&i\lambda \alpha_{n}\omega_{n1}+\PD_{1}\widetilde{R}_{n}&\cdots&i\lambda \alpha_{n}\omega_{nn}+\PD_{n}\widetilde{R}_{n}
\end{bmatrix}
.
\end{align}

Next we want to show that Det$A(t,x,\lambda)\neq 0$ almost everywhere in $Q_{T}$ for $\lambda\gg1$. 

\noindent
Using the fact that $\lVert \nabla_{t,x}\widetilde{R}_{j}\rVert_{L^{\infty}(Q_{T})}\leq C$ for some constant $C>0$ independent of $\lambda$, we have
\begin{align*}
\begin{aligned}
\lim_{\lambda\rightarrow \infty}\left\lVert \frac{{\rm Det}\, A(t,x,\lambda)}{\lambda^{3}}-\mbox{\rm Det}\begin{bmatrix}
	i\alpha_{0}&i\alpha_{0}\omega_{01}&\cdots&i\alpha_{0}\omega_{0n}\\
	i\alpha_{1}&i\alpha_{1}\omega_{11}&\cdots&i\alpha_{1}\omega_{1n}\\
	\vdots&\vdots&\cdots&\vdots\\
	i\alpha_{n}&i\alpha_{n}\omega_{21}&\cdots&i\alpha_{n}\omega_{nn}\\
	\end{bmatrix}\right\rVert_{L^{2}(Q_{T})}
	=0.
\end{aligned}
\end{align*} 
Therefore we have  that 
\begin{align*}
\begin{aligned}
\frac{{\rm Det}\, A(t,x,\lambda)}{\lambda^{3}}\rightarrow \mbox{\rm Det}\begin{bmatrix}
	i\alpha_{0}&i\alpha_{0}\omega_{01}&\cdots&i\alpha_{0}\omega_{0n}\\
	i\alpha_{1}&i\alpha_{1}\omega_{11}&\cdots&i\alpha_{1}\omega_{1n}\\
	\vdots&\vdots&\cdots&\vdots\\
	i\alpha_{n}&i\alpha_{n}\omega_{21}&\cdots&i\alpha_{n}\omega_{nn}\\
	\end{bmatrix}
	\neq 0\,\, \mbox{as}\ \lambda\rightarrow\infty \ \mbox{in}\  L^{2}(Q_{T}).
\end{aligned}
\end{align*}
Thus we can find a subsquence still denote the same such that 
\begin{align*}
\lim_{\lambda\to\infty} \frac{\mbox{Det}\, A(t,x,\lambda)}{\lambda^{3}}=\mbox{Det}\begin{bmatrix}
	i\alpha_{0}&i\alpha_{0}\omega_{01}&\cdots&i\alpha_{0}\omega_{0n}\\
	i\alpha_{1}&i\alpha_{1}\omega_{11}&\cdots&i\alpha_{1}\omega_{1n}\\
	\vdots&\vdots&\cdots&\vdots\\
	i\alpha_{n}&i\alpha_{n}\omega_{n1}&\cdots&i\alpha_{n}\omega_{nn}\\
	\end{bmatrix}\neq 0\,\,\mbox{pointwise for a.e.} \ (t,x)\in Q_{T}. 
\end{align*}
Hence we conclude that $\mbox{Det}\,A(t,x,\lambda)\neq 0$ for $\lambda\gg1$, a.e. $(t,x)\in Q_{T}$. Thus we have that $\nabla_{t,x}v_{0},\nabla_{t,x}v_{1},\cdots, \nabla_{t,x}v_{n}$ are linearly independent a.e. in $Q_{T}$. This completes the proof of Lemma \ref{Lemma about LI of solutions}. 
\end{proof}
\end{lemma}
Recall that \[\label{Final identity about vector b}\vec{b}(t,x)\cdot \nabla_{t,x}w(t,x)=0\,\,\mbox{for a.e.}\,\, (t,x)\in Q_{T} \ \mbox{and any solution $w$ to \eqref{Equation for w}}.\]Now using Lemma \ref{Lemma about LI of solutions}, we can choose $w_{0},w_{1},\ \cdots w_{n}$ solutions to \eqref{Equation for w} such that $\nabla_{t,x}w_{0},\nabla_{t,x}w_{1},\cdots,\nabla_{t,x}w_{n}$ are linearly independent for a.e. in $Q_{T}$. 
Using these choices of $w_{j}$ for $0\leq j\leq n$ in \eqref{Final identity about vector b}, we get 
$\vec{b}(t,x)=0, $ for a.e. $(t,x)\in Q_{T}$ but $\vec{b}\in C^{\infty}(Q_{T})$ therefore we have $\vec{b}(t,x)=0$ for all $(t,x)\in Q_{T}$. Hence $\vec{b}^{(1)}=\vec{b}^{(2)}$ in $Q_{T}$. This completes the proof of uniqueness for $\vec{b}$. }


\section*{Acknowledgement}
The second author would like to thank his Ph.D. supervisor Venky Krishnan for stimulating discussions. We also thank the several research funds which supported this study and they are as follows.  The first author was partially supported by Grant-in-Aid for Scientific Research (15K21766) of the Japan Society for the Promotion of Science for doing the research of this paper.  The  work of second author was partially supported by Grant-in-Aid for Scientific Research (15H05740 and 19K03554) of the Japan Society for the Promotion of Science for doing the research of this paper. He also  benefited from the support of Airbus Group Corporate Foundation Chair ``Mathematics of Complex Systems'' established at TIFR Centre for Applicable Mathematics  and  TIFR  International  Centre  for  Theoretical  Sciences,  Bangalore, India. 
The third author was partially supported by Grant-in-Aid for Scientific Research (19K03617) of the Japan Society for the Promotion of Science.


\end{document}